\documentclass{amsart}
\usepackage{geometry}
\usepackage{graphicx}
\usepackage{amssymb}
\usepackage{epstopdf}
\newcommand{\pf}{\mathcal{P}}
\DeclareMathOperator{\Tr}{Tr}
\newcommand{\R}{\mathbb{R}}
\newcommand{\T}{\mathbb{T}}
\newcommand{\so}{\mathfrak{so}}

\newtheorem{theorem}{Theorem}

\newtheorem{proposition}[theorem]{Proposition}
\newtheorem{corollary}[theorem]{Corollary}
\newtheorem{lemma}[theorem]{Lemma}
\newtheorem{conjecture}{Conjecture}

\title{Relative Equilibria of the 3-body Problem in $\mathbb{R}^4$}
\author{Alain Albouy, Holger R.~Dullin}

 \dedicatory{Dedicated to James Montaldi}

\subjclass{37N05, 70F10, 70F15, 70H33, 53D20.}
 \keywords{3-body problem, symplectic symmetry reduction, Lyapunov stability}

\email{alain.albouy@obspm.fr}
\email{holger.dullin@sydney.edu.au}
\begin{document}

\begin{abstract}
The classical equations of the Newtonian 3-body problem do not only define the familiar 3-dimensional motions. The dimension of the motion may also be 4, and cannot be higher. We prove that in dimension 4, for three arbitrary positive masses, and for an arbitrary value (of rank 4) of the angular momentum, the energy possesses a minimum, which corresponds to a motion of relative equilibrium which is Lyapunov stable when considered as an equilibrium of the reduced problem. The nearby motions are nonsingular and bounded for all time. We also describe the full family of relative equilibria, and show that its image by the energy-momentum map presents cusps and other interesting features.
\end{abstract}

\maketitle

\section{Introduction}

When considering the 3-body problem in 4-dimensional space additional relative equilibria appear,
and any triangle may appear as the shape of a relative equilibrium for certain masses.
Such unorthodox configurations were first described by Albouy and Chenciner and 
termed balanced configurations \cite{AlbouyChenciner98}.
The collinear configuration of Euler and the equilateral configuration of Lagrange appear as 
particular cases of such balanced configurations.

Denote the three masses by  $m_1,m_2,m_3$ and their position vectors  in $\R^4$ by $q_i$, $i=1,2,3$, 
with conjugate momenta $p_i \in \R^4$.
The Hamiltonian $H$ is the sum of kinetic energy $T$ and the Newtonian potential $V$:
\[
H = T + V, \quad
T = \frac12 \sum_{i=1}^3 \frac{  |p_i|^2 }{m_i}, \quad
V = -\sum_{i<j} \frac{ m_i m_j }{|q_i - q_j|} \,.
\]
Relative equilibria of the 3-body problem in $\R^4$ are equilibria in a certain rotating frame.
A rotating frame is determined by a rotation $R \in SO(4)$ with constant angular velocity 
$\Omega \in \so(4)$, i.e.{} $\Omega$ is a $4\times 4$ anti-symmetric matrix and 
$R = \exp( \Omega t)$. Thus set $q_i = \exp(\Omega t) \xi_i$ where $\xi_i$ are constant 
vectors in $\R^4$ and Newton's equations $m_i \ddot q_i = - \nabla_{q_i} V$ turn into the 
algebraic equations
\begin{equation} \label{eqn:releq}
   \Omega^2 \xi_1 =     \frac{ m_2 (\xi_2 - \xi_1)}{|\xi_2 - \xi_1|^3} + 
    \frac{ m_3 (\xi_3 - \xi_1)}{|\xi_3 - \xi_1|^3} 
\end{equation}
and two similar equations obtained by cyclic permutation of the indices.
In dimension 2 the matrix $\Omega^2$ is proportional to the identity,
and the equation reduces to the well known equation for central configurations,
where $\Omega^2$ is a negative scalar. The 3-dimensional case
reduces to the 2-dimensional case because in the gravitational problem 
all relative equilibria are planar.
No such reduction is possible in $\R^4$, and in general $\Omega^2$ is not 
proportional to the identity. The configuration of a relative equilibrium 
for which $\Omega^2$ is not necessarily proportional to the identity is called a 
balanced configuration \cite{AlbouyChenciner98}.

Thus the main  feature in the 4-dimensional problem that distinguishes it from the 3-dimensional 
problem is that the rotational symmetry group is $SO(4)$ instead of $SO(3)$.
Hence the angular velocity $\Omega$ 
and the angular momentum 
$L = \sum  q_i p_i^t - p_i  q_i^t = S \Omega + \Omega S$ 
are now anti-symmetric $4\times4$ matrices where $S = \sum m_i \xi_i \xi_i^t$ 
is the inertia tensor of the configuration.
The characteristic polynomial of an antisymmetric matrix $L$ may be written as 
$\det( L - \lambda I) = \lambda^4 + \lambda^2 \ell^2 + \pf^2$,
where  $\ell$ is the norm of $L$ and $\pf$ is its Pfaffian:
\[
     \ell^2 = \frac12 \Tr L L^t, \quad
     \pf^2 = \det L \,.
\]
Since the eigenvalues of an anti-symmetric matrix are on the imaginary axis we
have  $\ell^2 \ge 2 |\pf|$.  The eigenvalues of $\Omega^2$ are negative and at least double.

The equations that determine balanced configurations are invariant under the scaling 
$(\xi, \Omega) \to (s \xi, s^{-3/2} \Omega)$. 
As a result of this scaling energies scale with $s^{-1}$ and angular momentum scales with $s^{1/2}$, 
such that the product of the energy $H$ with squares of angular momenta is invariant under scaling.
Essential features of relative equilibria are described by their energy-momentum map.
Because of the scaling symmetry of the 3-body problem we choose 
two special scale-invariant combinations of the energy $H$ and the invariants 
of $L$:
\begin{equation} \label{eqn:hk}
  h=  H (\ell^2 + 2 |\pf|),\qquad k= \frac{|\pf|}{ \ell^2 + 2 |\pf|} \,.
\end{equation}
Since $\ell^2 \ge 2 |\pf|$ we have $0 \le k \le \tfrac14$.
The particular form of the scaling is motivated by the observation that for equilateral 
relative equilibria we have $h = h_L = const$, see the next section.
Denote the eigenvalues of $L$ by $\pm i \mu_1$, $\pm i \mu_2$ and choose  $\mu_1 > \mu_2 > 0$.
Then the scaling factor $\ell^2 + 2 |\pf|$ can be written as $(\mu_1 +  \mu_2)^2$.
The dimensionless ratio $k = \mu_1 \mu_2 / (\mu_1 + \mu_2)^2$ thus satisfies $ 0 \le k \le 1/4$.
Note that $k$ is a symmetric function of $\mu_1$, $\mu_2$, and hence the ordering of $\mu_i$ does not matter.

In this paper we describe the families of relative equilibria
for the 3-body problem in $\R^4$. Any motion of 3 bodies, with centre of mass at the origin, which does not remain in a fixed linear subspace $\R^3\subset \R^4$ has an angular momentum $L$ of rank 4. We prove that the Hamiltonian is bounded from below, and hence has a global minimum, on any level set of $L$ for which the fixed value of $L$ is of rank 4. Incidentally this shows that 
``Oldest problem in dynamical systems'' as formulated by Herman in his 1998 ICM 
address  \cite{Herman98}, also see \cite{ACS12},  has the answer ``No'' in dimension 4: 
Near the minimum of the Hamiltonian no initial condition is unbounded, 
see Theorem~\ref{thm:sta} below. By contrast, in the $n$-body in $\R^3$, $n\geq 3$, the rank of $L$ is 0 or 2, and $H$ does not even possess local minima on a level set of $L$. The importance of this fact was recently emphasised by \cite{Scheeres12} and \cite{Moeckel17}.
In \cite{DS19}  the symplectic reduction for the 3-body problem in $\R^4$ is described. 
In that paper it is shown that all three families of balanced configurations
are local minima of the reduced Hamiltonian near 
their collision limit when $k$ is sufficiently small. By contrast, in the present paper
we show that for any fixed $k$ (not only sufficiently small) there is a globally minimal 
balanced configuration.
Using the globally symmetry reduced 3-body problem with the Poisson structure described  in \cite{Dullin13} the stability and minimality of the Hamiltonian along all families of balanced configurations will be studied in a forthcoming paper.

\section{Equilateral solutions (any masses)}

Say we have an equilateral solution with $|\xi_i - \xi_j | = r$ then 
by scaling symmetry we can choose $r = 1$ and \eqref{eqn:releq}
becomes linear in $\xi_i$ for given $\Omega^2$. Complementing 
the equations by $\sum m_i \xi_i = 0$ shows that the solvability 
condition is that both frequencies are equal to $m_1 + m_2 + m_3$ 
and hence these solutions are central configurations. 
A particular solution can be constructed in a 2-dimensional subspace by 
picking any equilateral triangle with centre of mass at the origin.
There is a family of such solutions because given that $\Omega^2$ 
is proportional to the identity, there is a family of matrices $\Omega$
whose square is the identity: any anti-symmetric matrix of the form 
\[
    J = \begin{pmatrix}
    0    & -u_1 & -u_2 &  -u_3 \\
    u_1 & 0     & \phantom{-}u_3 & - u_2 \\
    u_2 & -u_3 & 0   & \phantom{-}u_1 \\
    u_3 & \phantom{-}u_2 & - u_1 & 0
    \end{pmatrix}
\]
with $u_1^2 + u_2^2 + u_3^3 = 1$ satisfies $J^2 = -I$, it defines a complex structure.
Thus we can set $\Omega = \omega J$ for some scalar $\omega$. 
This leads to a characterisation of all equilateral relative equilibria, related to a characterisation obtained by Chenciner \cite{Chenciner11}:

\begin{theorem} \label{thm:equilateral}
For the family of equilateral solutions we have 
\[
h = h_L = -\tfrac12 M_2^3/ M, \quad
\text{and} \quad
k \in [0, \tfrac34 M_3 M / M_2^2]
\]
where the symmetric functions of the masses are
\[
M = m_1+m_2 +m_3, \quad M_2 = m_1 m_2 + m_2 m_3 + m_3 m_1, \quad M_3 = m_1 m_2 m_3 \,.
\]
\end{theorem}

The curve described in this theorem is shown as a vertical black line in 
Figs.~\ref{fig:gen1}--\ref{fig:gen23}.
The maximal possible $k$ is $1/4$ when the masses are all equal,
while the maximal possible $k$ goes to zero when any mass goes to zero.
The scaled energy $h_L \le 0$ goes to zero when two masses go to zero, 
while the minimum $h_L = -9 m^5/2$ is found when all masses are equal.

\begin{proof}
Assuming we are in a coordinate system in which the tensor of inertia $S$ is diagonal, 
and since the triangle is planar two eigenvalues are equal to zero, the other two we call $\Theta_1, \Theta_2$. 
With such $S$ and the angular velocity  $\Omega = \omega J$ one computes that 
$\det L = ( \Theta_1 \Theta_2 ( u_2^2 + u_3^2) \omega^2 )^2$ 
and $\ell^2 = (\Theta_1^2 + \Theta_2^2 + 2 \Theta_1 \Theta_2 u_1^2) \omega^2$ and hence 
$\ell^2 + 2 \pf =  (\Theta_1 + \Theta_2)^2 \omega^2$ which is independent of $J$.
Notice that $H$ only depends on $\Omega^2$ which is 
independent of $J$, and hence $h$ will remain constant when $J$ is changed.
To compute the value of $h$ observe that the potential is $-M_2/r$ 
and the kinetic energy is $\frac12 (\Theta_1 + \Theta_2)\omega^2$. Finally the condition  
to have a relative equilibrium is $\omega^2 = M / r^3$,
so that $T = -V/2$.
For masses $m_1, m_2, m_3$ at the corners of an equilateral triangle with 
sides of length $r$ the tensor of inertia with respect to the centre of mass has eigenvalues $\Theta_1, \Theta_2$ 
that satisfy $4\Theta_1 \Theta_2 = 3 r^4 M_3/M$ and $\Theta_1 + \Theta_2 = r^2  M_2 /M $.
Thus the result follows.
\end{proof}

The image of this family is an interval in $k$, and thus trivially a convex polytope. 
In \cite{Chenciner13} it is shown that this is true in any dimension.
\emph{ We observe that for all relative equilibria that we found (see next sections) it is true that $h \le h_L$.}

\begin{figure} 
\includegraphics[width=6.5cm]{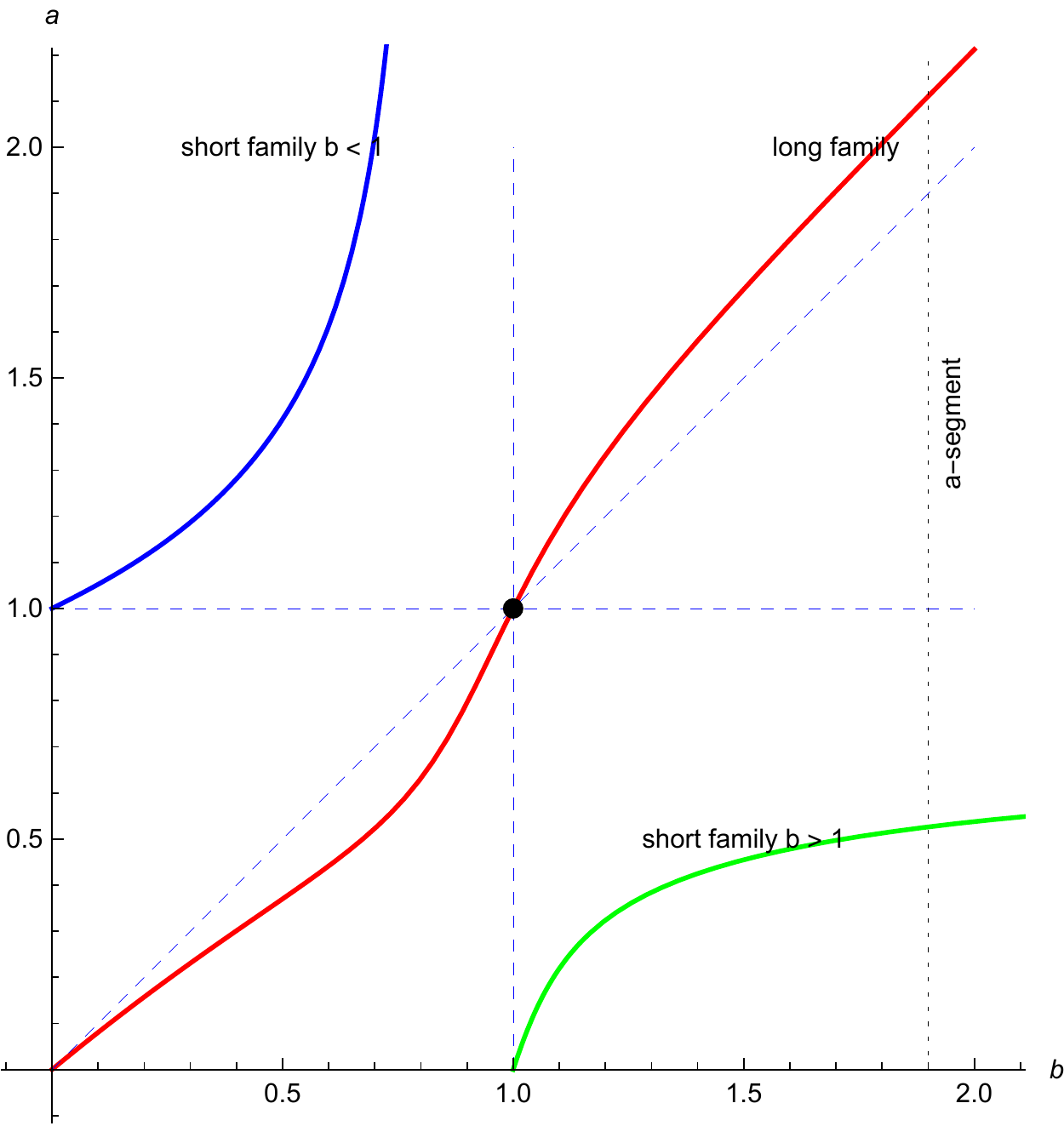} \hspace*{0.5cm}
\includegraphics[width=5cm]{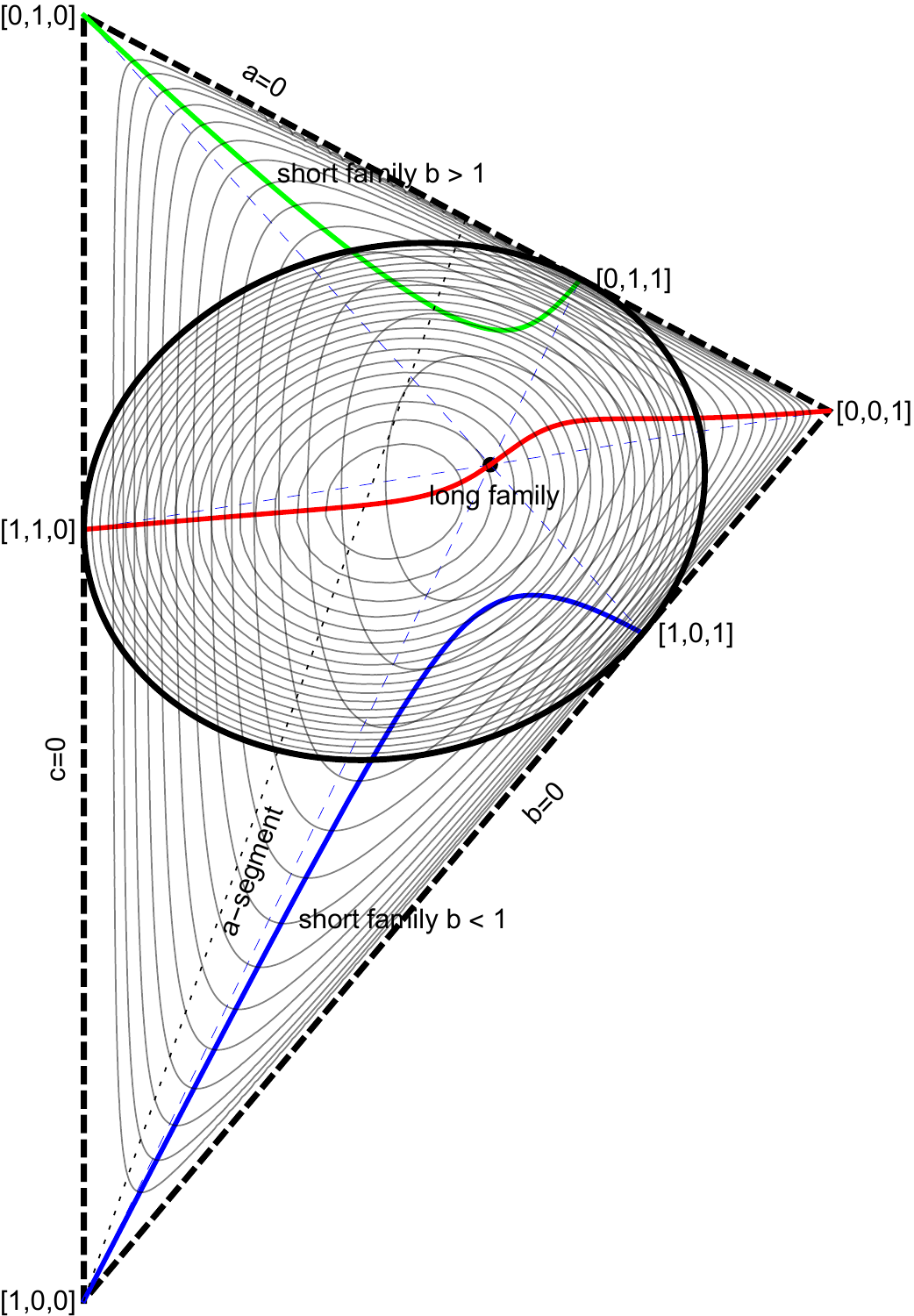}
\caption{Three smooth families of balanced configurations. Long family red, short families blue and green. Masses $(m_1, m_2, m_3) = (3,2,1)/6$.
Isosceles shapes are shown as dashed blue lines.
Left: $a(b)$ for $c=1$, the long family exists for all values of $b$. 
Right: The extended  triangle of shapes $I = const$ with boundary black dashed where one side length vanishes. 
The thick black ellipse marks shapes with area $A=0$ with contour lines of constant positive area inside. 
The other set of contour lines indicate $V = const$. Special points are marked by their projective triple $[a,b,c]$.}
\label{fig:shapetriangle}
\end{figure}

\section{Balanced configurations}

Denote the distances between the particles by $d_{ij} = |\xi_i - \xi_j|$, which are constant for relative equilibria.
Write $a=d_{23}^2$, $b=d_{13}^2$, $c=d_{12}^2$. 
The equation of balanced configurations  \cite{AlbouyChenciner98,Albouy04} is $B=0$ where
\begin{equation}\label{eqn:B}
   B=\begin{vmatrix} 1&1&1 \\ 
      m_1(b+c-a)&m_2(c+a-b)&m_3(a+b-c)\\
       a^{-3/2}&b^{-3/2}&c^{-3/2} 
   \end{vmatrix}.
\end{equation}
The moment of inertia, the squared area of the triangle, and the potential of a relative equilibrium are
\begin{align*}
 I = & \frac{ m_1 m_2 c + m_2 m_3 a + m_3 m_1 b }{m_1 + m_2 + m_3}, \\
 A^2 = &\frac{1}{16} (2 a c + 2 a b + 2 b c - a^2 - b^2 - c^2), \\
 V = & -m_1 m_2 c^{-1/2} -  m_2 m_3 a^{-1/2} -  m_3 m_1 b^{-1/2}
\end{align*}
where the area $A$ is determined by Heron's formula. We see immediately that the equilibrium condition $B = 0$ is obtained 
by requiring that the gradients of these three functions are linearly dependent.
In \cite{AlbouyChenciner98}, Prop.\ 2.25, this property is not deduced directly from
the dynamical property. An important intermediate step is needed: the equation of balanced configurations expressed as the commutation of two linear operators (Prop.\ 2.6).

Because of homogeneity we can restrict to $c=1$. The solutions of $B=0$ are shown 
as graphs $a(b)$ in Fig.~\ref{fig:shapetriangle}, left.
We will prove below that there are always three smooth solution families, 
and for $m_1 > m_2 > m_3$ they look as in the figure. 
The solution branch starting at the origin and going through the Lagrange point $a = b = c$ we call the long family,
the other two are the short families.
To illustrate the solutions to $B=0$ we restrict to the plane $I = const$, and plot the level 
lines of $A$ and $V$ on this plane, see Fig.~\ref{fig:shapetriangle}, right.
This presentation in the  triangle of shapes goes back to \cite{AlbouyChenciner98}.
The restriction of $I = const$ to the positive octant is the triangle of shapes shown in green, see Fig.~\ref{fig:shapetriangle},
whose edges correspond to the collisions.
Along these boundaries  $V \to -\infty$ and so the contour lines of the potential hug the edges of the triangle of shapes.
Within the triangle is an ellipse where the area vanishes, $A = 0$, which is a fat black curve tangent to 
the edges of the triangle of shapes in Fig.~\ref{fig:shapetriangle}, with the contour lines of the area $A$ limiting onto this ellipse.
Inside this ellipse the triangle inequality is satisfied.
Each function $A$ and $V$ (restricted to the plane $I = const$) has a critical point, 
and the long family  (shown in red) passes through these critical points.
The maximum of $V$ is at the equilateral Lagrange configuration $ a=b=c$, 
while  the maximum of $A$ is at the configuration with a round tensor of inertia given by
$a = m_1(m_2+m_3)$, $b=m_2(m_1+m_3)$, $c=m_3(m_1+m_2)$.

Solution curves that are outside the ellipse with $A=0$ are un-physical solutions, but we show 
them in any case because the three solution families all limit to the corners of the extended triangle of shapes
where two side-lengths are vanishing, indicated by $[0,0,1]$ etc. 
Solution curves that emerge at the tangency of the ellipse $A=0$ and the outer boundary of the 
triangle of shapes correspond to collisions where a single side-length vanishes, indicated by $[0,1,1]$ etc.
A crossing of a solution curve with the ellipse $A=0$ away 
from the outer boundary of the triangle of  shapes corresponds to degenerate shapes with area zero, 
i.e.\  Euler collinear solutions. Isosceles shapes are shown as thin blue lines.

If the three masses $m_1$, $m_2$, $m_3$ are distinct, we have three families which do not intersect, 
see Fig.~\ref{fig:shapetriangle}.
The same is true if two masses are equal and these masses are strictly greater than the third mass, 
see Fig.~\ref{fig:2equ12}.
If there is a pair of equal masses, one of the curves of balanced configurations is a segment of straight line, 
which contains isosceles configurations only. 
If the three masses are equal, the three curves are segments of straight lines, see Fig.~\ref{fig:equalmasses}.
They are all the isosceles triangles, and they intersect at the equilateral triangle. 
In the remaining case two masses are equal, and they are smaller than the third mass, see Fig.~\ref{fig:2equ2}.
The long family is a segment of a straight line made of isosceles triangles. 
Only in this case do the short families cross the long family at two distinct points.
These statements will now be proved. 
The figures display other features,  for example inflection points, which we will not discuss. 
We will prove two lemmas giving these proofs and a bit more.

The lines where $b$ and $c$ are fixed and $a$ is changing project as follows on the triangle of shapes. 
At $a=0$, the point is on the side of the triangle of shapes with equation $a=0$, the top line in 
Fig.~\ref{fig:shapetriangle}, connecting $[0,1,0]$ and $[0,0,1]$.
When $a$ is growing, the point follows a straight line. 
The slope of this line is determined by the ratio of $b$ and $c$. 
When $a=+\infty$, the point is at the vertex $[1,0,0]$ of the triangle of shapes which is opposite to the side $a=0$. 
We call such a segment an {\it $a$-segment}.
When the $a$-segment moves the roots of $B=0$ draw the curves of balanced configurations, 
which we study here in the whole triangle of shapes, not only inside the ellipse of true triangles.

\begin{lemma}\label{lem:A}
Along an $a$-segment, there are at most two roots of the equation $B=0$, except in the case $b=c$ and $m_2=m_3$, for which $B=0$ identically along the $a$-segment.
\end{lemma}
\begin{proof}
We compute
\[
\frac{d^2B}{da^2}=-3 \begin{vmatrix} 1&1&1\cr -m_1&m_2&m_3\cr a^{-5/2}&0&0 \end{vmatrix} +
\frac{15}{4} \begin{vmatrix} 1&1&1\cr m_1(b+c-a)&m_2(c+a-b)&m_3(a+b-c)\cr a^{-7/2}&0&0 \end{vmatrix} 
\]
$$=\frac{3}{4}(m_3-m_2)a^{-5/2}+\frac{15}{4}a^{-7/2}(m_2+m_3)(b-c)=\frac{3}{4}a^{-7/2}\bigl(a(m_3-m_2)+5(b-c)(m_2+m_3)\bigr).$$
If $m_3-m_2$ and $b-c$ have the same sign, or if one vanishes and not the other, $d^2B/da^2$ has constant sign, so that the function $B$ has at most two roots. If $m_3-m_2$ and $b-c$ have opposite signs, then $d^2B/da^2$ changes sign once, and $B$ could have up to three roots. But $B$ is asymptotically $(m_2+m_3)(b-c)a^{-3/2}$ near $a=0$ and  $(m_1+m_2)(c^{-3/2}-b^{-3/2})a+(m_2-m_3) b^{-3/2} a$ near $a=+\infty$. We see that the sign is the same at these two limits: this is the sign of $b-c$ which is by hypothesis also the sign of $m_2-m_3$.
So, there cannot be three roots, two is the maximum allowed. 
\end{proof}

This is essentially Lemma 4.3 of Albouy-Chenciner, \cite{AlbouyChenciner98}. Then, the authors claim that one can describe the balanced configuration curves with this lemma. Let us see this here, by proving another lemma.

\begin{lemma}\label{lem:B}
Suppose that $m_2$ and $m_3$ are the smallest masses. 
An $a$-segment with $b\neq c$ passes through two isosceles triangles.
When $a$ is growing from $a=0$, there is exactly one root before, or at, the first isosceles triangle. There are no roots between the two isosceles triangles.  There is at most one root after or at the second isosceles triangle. If $b=c$ and $m_2\neq m_3$ the equilateral solution is the only root.
\end{lemma}

\begin{proof}
If $a=b$ then $B=(m_2-m_1)c(c^{-3/2}-b^{-3/2})$. If $a=c$, $B=(m_3-m_1)b(c^{-3/2}-b^{-3/2})$. These two values have the same sign. At $a=0$ the asymptotics is
$(m_2+m_3)(b-c)a^{-3/2}$ as seen in the previous proof, which is of opposite sign. We conclude easily using Lemma~\ref{lem:A}. 
\end{proof}

{\bf Remark.} According to the hypothesis of Lemma~\ref{lem:B} and the asymptotics of $B$ as $a\to\infty$, the second root on the $a$-segment does exist except if 
\begin{equation}\label{eqn:O}
\Bigl(\frac{m_1+m_3}{m_1+m_2}\Bigr)^{2/3}<\frac{b}{c}<1\qquad\hbox{or if}\qquad
1<\frac{b}{c}<\Bigl(\frac{m_1+m_3}{m_1+m_2}\Bigr)^{2/3},
\end{equation}
the first case requiring $m_3<m_2$, the second $m_2<m_3$.
\medskip

The description of the curves is now easy. Essentially, we proved that when the $a$-segment sweeps out the triangle of shapes, it draws curves which can only meet each other when $b=c$. The curves are smooth since the roots cannot degenerate and the implicit function theorem can consequently be applied. Then, it is enough to consider the Lagrange equilateral triangle and the roots on the boundaries of the triangle of shapes to see how they should connect. The curve passing through the Lagrange point may be continued both sides and gives what we call the long family.
Thus we have proved
\begin{theorem}
The set of balanced configurations in the triangle of shapes is always the union of three smooth curves, with three pairs of ends which are the three vertices of the triangle of shapes and the three points on the sides corresponding to two equal mutual distances and a zero mutual distance.
\end{theorem}

The equations of the Euler configurations coincide with the equations of balanced configurations restricted to the flat triangles. The three above curves cannot cross the ellipse of flat triangles except at the three Euler points.

\begin{corollary}\label{cor:T}
The set of balanced configurations which consist of true triangles, i.e., which are inside the ellipse of flat triangles, is always the union of three smooth curves, with three pairs of ends which are the three Euler configurations and the three points corresponding to two equal mutual distances and a zero mutual distance.
\end{corollary}

Now we are going to show that along a family $a(b)$ as determined by the $a$-segment 
with constant $c$ there is monotonicity. For the proof we also need to consider 
the $b$-segment. The $b$-segment is a line segment starting at 
$[0,1,0]$ and ending somewhere on the opposite side $[a,0,c]$. 
Fixing $c=1$ the root $b(a)$ is determined as the zeroes of $B$ 
along the $b$-segment. These roots can also be seen in Fig.~\ref{fig:shapetriangle}.

\begin{lemma} \label{lem:mono}
When $m_1 > m_2 > m_3$ then along a family of balanced configuration with $c=1$ the function $a(b)$ is increasing with non-zero derivative.
\end{lemma}
\begin{proof}
The asymptotic values of $B$ for $b\to 0$ and $b\to \infty$ either have the same 
sign or opposite signs. 
In the case where $((m_1+m_3)/(m_1+m_2))^{3/2} < b/c < 1$ they have 
opposite signs, there is a single root along the $a$-segment, and $dB/da$ only takes negative values. For sufficiently large or small values of $b/c$ there are two roots along the $a$-segment, which are separated by a root of $dB/da$.
Combining with the asymptotic behaviour already described in 
Lemma~\ref{lem:A} this implies that along the long family 
$dB/da < 0$ while along both short families $dB/da > 0$.

Observe that analogues of Lemma~\ref{lem:A}, conditions (\ref{eqn:O}) and the previous
paragraph hold for the $b$-segment as well. 
For the $b$-segment there is a single root when $((m_2+m_3)/(m_2+m_1))^{3/2} < a/c < 1$ and the derivative $dB/db$ only takes positive values.
Now $B$ is asymptotically $(m_1+m_3)(a-c) b^{-3/2}$ near $b = 0$ and 
$(m_1+m_2)c^{-3/2}b  - (m_2+m_3)a^{3/2}b$ near $b = +\infty$.
Thus along the long family $dB/db > 0$ while along both short families $dB/db < 0$.

Now the derivative along a family can be computed by implicit differentiation as
$da/db = -(dB/db)/(dB/da)$. We have just shown that the signs of the derivatives
of $B$ in the $a$- and the $b$-direction have opposite signs along all three families,
and hence the result follows.
\end{proof}

\section{Minima of the Hamiltonian}

Now we are going to prove that the Hamiltonian is bounded below when the angular momentum matrix $L$ 
is of rank 4.

\begin{proposition}\label{pro:bounded} 
Given three positive masses $m_1, m_2, m_3$ and a $4\times 4$ antisymmetric matrix $L$ of rank 4, consider the 3-body problem in $\mathbb{R}^4$ with these masses, and consider in the phase space the submanifold of states with angular momentum $L$. On this submanifold the energy $H$ is bounded below by a negative number $H_0$.
\end{proposition}

{\bf Remark.} Recall that a motion of the 3-body problem has a dimension which is one (rectilinear motion), two (planar, non-rectilinear motion), three (spatial, non-planar motion) or four (non-spatial motion). The hypothesis about the rank of $L$ is motivated by this observation: In the 3-body problem, the motion is of dimension 4 if and only if the angular momentum $L$ is of rank 4.
\medskip

{\bf A preliminary lemma about relative equilibria.} To each balanced configuration (with a given size) is associated a motion of relative equilibrium which may be described as follows. In this first discussion, we do not need to consider the exceptional case of the configuration with round ellipse of inertia. The configuration uniformly rotates around each of its axis of inertia. There is a decomposition of $\R^4$ of the form $\R^2\oplus\R^2$ such that each axis of inertia rotates uniformly in each $\R^2$, one with angular velocity $\omega_1$, the other with angular velocity $\omega_2$. In this frame, the matrix  in (\ref{eqn:releq}) is
\[
    \Omega = \begin{pmatrix}
    0    & -\omega_1 & 0 &  0 \\
    \omega_1 & 0     & 0 & 0 \\
   0 & 0 & 0   & -\omega_2 \\
    0 & 0 & \omega_2 & 0
    \end{pmatrix}.
\]
The values of $\omega_1$ and $\omega_2$ are uniquely defined,  up to sign, by equation (\ref{eqn:releq}).

\begin{lemma}\label{lem:C}
The dimensionless momentum $k$ goes to zero at both endpoints of each of the three curves of balanced configurations described in Corollary~\ref{cor:T}.
\end{lemma}

\begin{proof} Let the configurations satisfy the normalisation $I=1$. We denote by $\Theta_1$ and $\Theta_2$ the moments of inertia of the triangle, satisfying $I=\Theta_1+\Theta_2$. According to Corollary~\ref{cor:T}, we may arrive at a collinear Euler configuration or at a configuration with one distance going to zero.  In both cases, the squared area $A^2=M\Theta_1\Theta_2/4M_3$ tends to zero. An axis of inertia is such that the max $\epsilon$ of its three distances to each of the three bodies  tends to zero. We claim that the angular momentum $\mu_1$ concerning the rotation around this axis tends to zero, while the other angular momentum $\mu_2$ remains finite and bounded away from zero. This implies $k\to 0$. Let us assume that $\mu_1$ does not tend to zero while $\epsilon\to 0$. Then, for a sequence of configurations going to an Euler configuration or a collision, the angular velocity $\omega_1$ is of order $\epsilon^{-2}$, since $\mu_i=\omega_i\Theta_i$ and since $\Theta_1$ is of order $\epsilon^2$.  The left-hand side of  (\ref{eqn:releq}) for the first $\R^2$ would be of order $\epsilon^{-3}$. But we claim that the right-hand side can be at most of order $\epsilon^{-2}$.  Indeed, we need a small binary to get a big right-hand side. If this small binary is above the axis, the third body is below since the centre of mass is on the axis. This contradicts the equation $m_1x_1y_1+m_2x_2y_2+m_3x_3y_3=0$ satisfied by the coordinates $(x_i,y_i)$ of the bodies in the axes of inertia. So the axis passes between the two bodies of the small binary. Due to the centre of mass condition, at least two of the three $y_i$'s are of order $\epsilon$. Consequently the smallest mutual distance is of order at least $\epsilon$. The right-hand side of (\ref{eqn:releq}) is of order at most $\epsilon^{-2}$ and we have a contradiction.
\end{proof}

{\bf Remark.} See \cite{DS19} for precise estimates on relative equilibria at these endpoints.
\medskip

{\bf Some classical formulas.} We introduce the Jacobi vectors (their main property is already presented by Lagrange in his {\it M{\'e}chanique Analitique} \cite[p.\ 292]{Lagrange1788})
$$q=q_2-q_1,\quad Q=q_3-(m_1 q_1+m_2 q_2)/(m_1+m_2)$$
We always assume $m_1\dot q_1+m_2\dot q_2+m_3\dot q_3=0$. Twice the kinetic energy is $$2T=\frac{m_1m_2}{m_1+m_2}\|\dot q\|^2+\frac{m_3(m_1+m_2)}{ m_1+m_2+m_3} \|\dot Q\|^2$$
We set $$\mu=\frac{m_1m_2}{m_1+m_2},\quad \nu=\frac{m_3(m_1+m_2)}{m_1+m_2+m_3},\quad p=\mu\dot q,\quad P=\nu\dot Q.$$
The Hamiltonian is now
$$H=\frac{|p|^2}{2\mu}+\frac{|P|^2}{2\nu}-\frac{m_1m_2}{ d_{12}}-\frac{m_2m_3}{ d_{23}}-\frac{m_1m_3}{ d_{13}}.$$
A short  computation gives the following expression of the angular momentum
\begin{equation} \label{eqn:1}
   L = qp^t - pq^t + QP^t - P Q^t.
\end{equation}
We can of course also express $L$ with wedge products
\begin{equation} \label{eqn:2}
L=\mu q\wedge \dot q+\nu Q\wedge \dot Q,
\end{equation}
and consider the space of $4\times 4$ antisymmetric matrices as the space of bivectors,
endowed with the usual Euclidean form, which is half of the Euclidean form on the space of matrices, and which is such that 
\begin{equation} \label{eqn:3}
|q\wedge\dot q|^2= |q|^2|\dot q|^2-\langle q,\dot q\rangle^2.
\end{equation}

{\bf A standard estimate.} By hypothesis $L$ is of rank 4. But $\mu q\wedge\dot q$ and $\nu Q\wedge\dot Q$ are of rank 2. Given an $L$ of rank 4, the two terms in (\ref{eqn:2}) are bounded away from zero. There is a $d_L>0$ such that $|\mu q\wedge\dot q|\geq d_L$. As a simple-minded and non-optimal choice, we may choose this $d_L$ as being the distance in the space of bivectors between $L$ and the set of bivectors of rank 2, which contains $\nu Q\wedge\dot Q$. Combining \eqref{eqn:3}, $d_{12}=|q|$ and $p=\mu\dot q$ we get
\begin{equation} \label{eqn:4}
|p|\geq \frac{d_L}{d_{12}}.
\end{equation}

\begin{proof} (\emph{of Proposition~\ref{pro:bounded}})
If $H$ were not bounded below, there would exist a sequence of states in the submanifold where one of the three negative terms in $H$ tends to $-\infty$. By extraction, there would be another sequence such that the most negative term of the three is always the same term.  By renumbering the bodies, we say that $-m_1m_2/d_{12}$ is this term.  But $H$ has the term  $|p|^2/2\mu$ where $|p|^2\geq d_L^2/d_{12}^2$ goes to $+\infty$ faster than the negative terms. So, $H$ goes to $+\infty$, which is a contradiction. 
\end{proof}

{\bf Remark.} There is an orthonormal frame $(e_1,e_2,e_3,e_4)$, a $\mu_1>0$ and a $\mu_2>0$ such that $L=\mu_1 e_1\wedge e_2+\mu_2 e_3\wedge e_4$. The optimal choice for $d_L$ is well defined as $\min (\mu_1,\mu_2)$.
\medskip

Recall that an {\it integral manifold} is a submanifold of the phase space obtained by fixing the first integrals. On a submanifold defined by fixing $L$ to a rank 4 value, the minima of the energy $H$ correspond to isolated balanced configurations: If not, we would have by analyticity all a curve of balanced configurations with same $L$, which would contradict Lemma~\ref{lem:C}. By fixing $L$ and then $H$ slightly above such a minimum, we get the following result. 

\begin{theorem}\label{thm:sta}
There are compact integral manifolds in the 4-dimensional 3-body problem.  
There are relative equilibria whose motion is 4-dimensional 
which are nonlinearly stable in the fully reduced system. 
\end{theorem}

{\bf Related results.} From estimate \eqref{eqn:4} we deduce as easily that the distance $d_{12}$ is bounded away from zero on any integral manifold of the 3-body problem defined by a rank 4 value of the angular momentum $L$ and an arbitrary value of the energy $H$. As a corollary, we get:

\begin{proposition}
There are no collisions in a motion of the 3-body problem which is 4-dimensional.
\end{proposition}

The no-binary-collision result is related to a no-syzygy result which is easier: the latter implies the first in the case of a 3-body problem with a potential which is non-singular at the collision. A {\it syzygy} is a state of the three body problem with the three bodies on a same line. In Jacobi vectors, a syzygy is characterised by the proportionality $q=\lambda Q$ for some $\lambda\in \mathbb{R}$. Considering expression \eqref{eqn:2}, we see that $Q$ factors out and $L$ is then of rank 2. So

\begin{proposition}
There are no syzygies in a motion of the 3-body problem which is 4-dimensional.
\end{proposition}

These results may be compared to classical results about the 3-body problem in the usual dimensions (see \cite[\S 326--7]{Wintner41}). If the angular momentum is zero, the motion is planar or rectilinear. In a 3-dimensional motion, the angular momentum is a non-zero vector. The plane orthogonal to it is called the invariable plane. {\it In a 3-dimensional motion, the three bodies are all in the invariable plane if and only if there is a syzygy.}

The no-collision result happens in other situations, where the energy is also bounded below. The isosceles 3-body problem in dimension 3 with non-vanishing angular momentum along the symmetry axis is an example. The 4-body problem in dimension 6 and the 5-body problem in dimension 8 are other examples. Another classical result was stated by Weierstrass and proved by Sundman (see \cite{Sundman12}): if in the 3-body problem there is a triple collision, then the angular momentum is zero.

\begin{figure} 
\includegraphics[height=6cm]{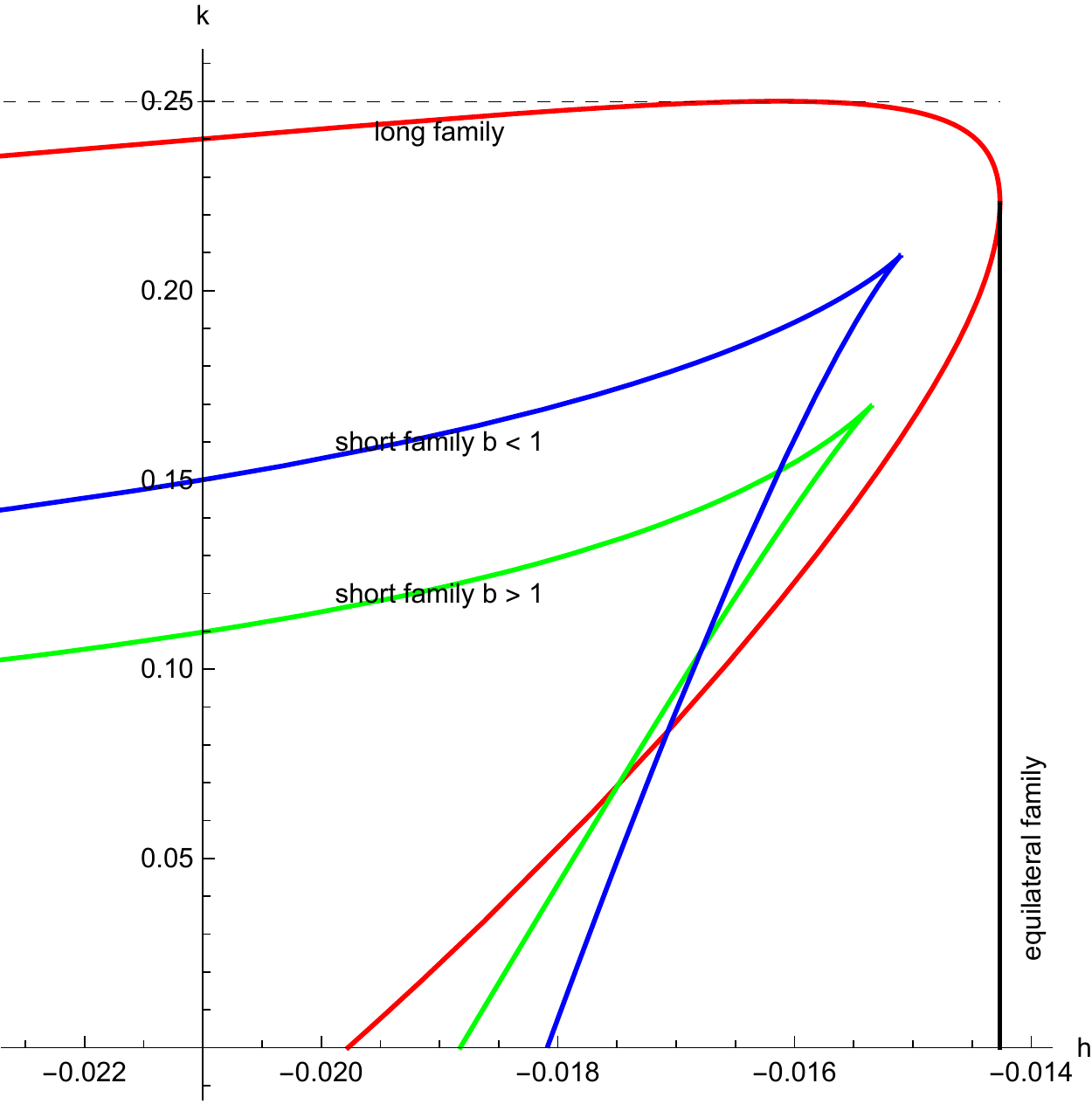}
\includegraphics[height=6cm]{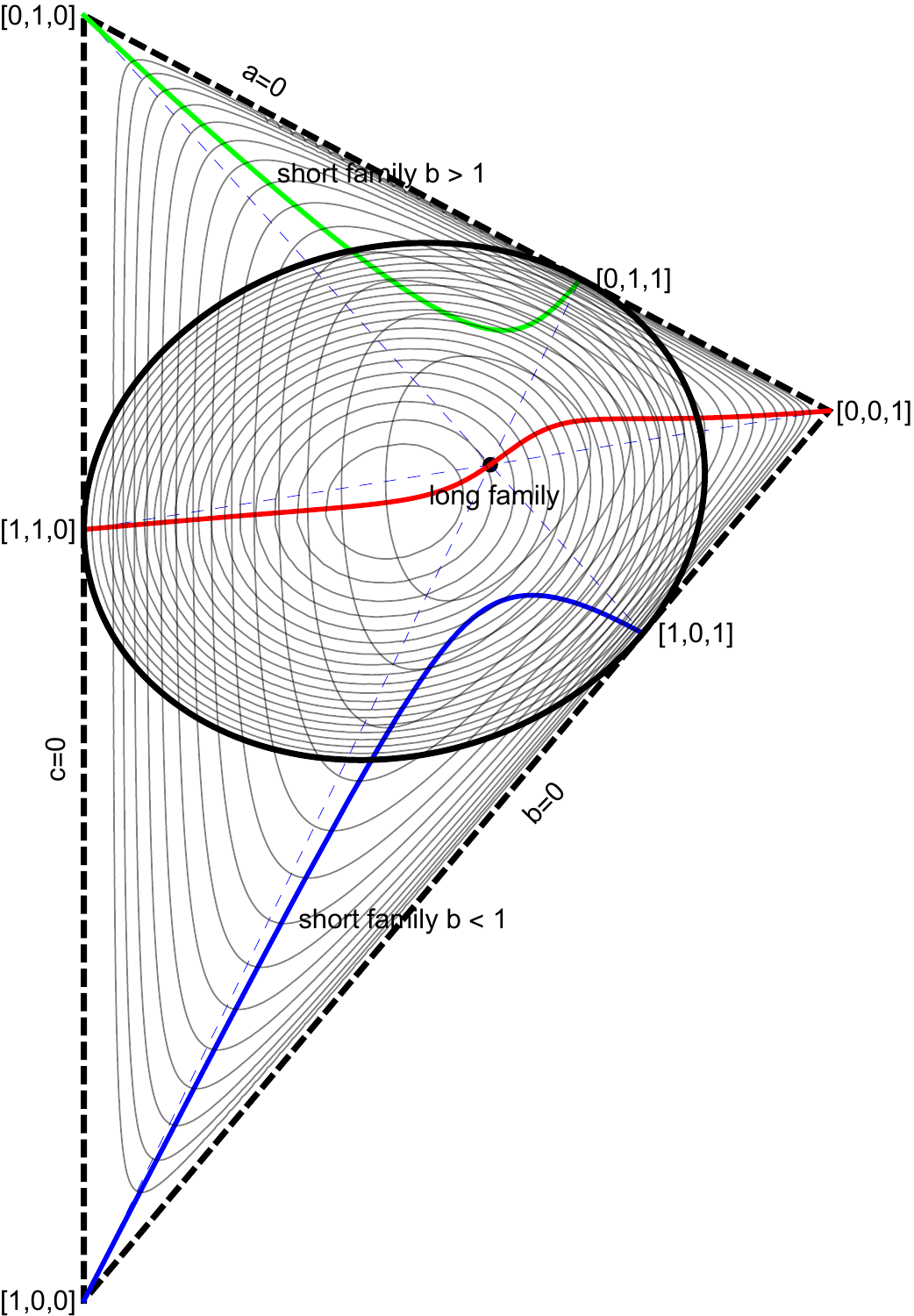}
\caption{Three distinct masses $(m_1, m_2, m_3) = (3,2,1)/6$.
The Lagrange equilateral family is the vertical line, not extending 
all the way to $k = 1/4$. 
The three non-equilateral families emerge from Euler's collinear configurations at $k=0$
and for $h \to -\infty$ approach collision configurations.
The two short families have a cusp each.
The long family emerges at the Euler collinear configuration with the smallest energy, 
then touches the endpoint of the equilateral family, and then is tangent to the maximum at $k=1/4$.
Past this tangency it corresponds to minimal energy at fixed $k$ and hence is non-linearly stable.
} \label{fig:gen1}
\end{figure}

\section{Frequencies in the energy-momentum map}

Relative equilibria are best presented in an energy-momentum map. 
Because of the scaling symmetry of the 3-body problem it is useful to 
consider a scaled energy that is invariant under the scaling, and also 
a combined momentum invariant that is scale invariant, as was introduced
in \eqref{eqn:hk}. In general when considering an energy-momentum map 
the slope of the graph of the energy as a function of the conserved momentum 
gives the frequencies. In our case this relation is somewhat unusual,
and is the content of this section. The expression for the slope of the
$(h,k)$-curves will allow us to deduce some consequences for the structure
of these curves.

Before we prove the formula for $dk/dh$ along a family of relative equilibria
let us describe the image of the scaled energy-momentum map as obtained from a numerical study. Part of this description will be later confirmed by general mathematical statements. We are considering the generic case of distinct masses with 
$m_1 > m_2 > m_3$.

Along each of the three smooth families the curve  \eqref{eqn:hk} is computed. 
The resulting three curves are shown in Fig.~\ref{fig:gen1}, 
with the same colours for the long and short families as in Fig.~\ref{fig:shapetriangle}.
Only parts of the families inside the ellipse $A=0$ have a point in the image. 
Each family has a point with $A=0$, the collinear Euler configuration, 
which has $k=0$, because the Pfaffian vanishes for 3-dimensional 
momentum $L$.

The short families have larger energy at the initial point as Euler collinear configurations. 
Each short family has a cusp in the energy-momentum map. 
Past the cusp the solution branch approaches $k =0 $  again while $h \to -\infty$. 
The long family starts out as the Euler configuration with the smallest energy and 
extends all the way to the energy of the Lagrange central configurations. There the 
family has a maximum in $h$. Simple considerations on the inverse problem prove that from there the balanced configurations of this family are acute triangles. Following the family there is a maximum at 
$k = 1/4$. This is the absolute maximum possible for $k$.
Past the tangency with $k = 1/4$ the long 
family is the global minimum of the Hamiltonian (for fixed $k$).

Cases with equal masses are limiting cases of generic situation, but the limits are non-trivial 
since some branches coalesce. In the most symmetric case with all masses equal 
there is only a single solution branch left. These special cases are described in the next section.

In order to prove the formula for $dk/dh$ we first need a result about Darboux coordinates
such that we can conclude that $\partial H/ \partial \mu_i = \omega_i$.
We consider all the balanced configurations (with all the sizes) and embed each of them  in $\R^2 \oplus \R^2$ as a relative equilibrium: an axis of inertia in the first $\R^2$, the other in the second, and the unique velocities making two counterclockwise uniform rotations that ``balance'' the gravitational attraction and maintain each axis of inertia in its respective plane $\R^2$.

In the particular case of an inertia matrix proportional to the identity, a pair of orthogonal axes which is fixed with respect to the configuration acts as the uniquely defined pair of axis of inertia of the general case.

Unconstraining the size of the configuration, the three smooth curves of Corollary~\ref{cor:T} define three smooth 2-dimensional semi-cones $\mathcal{B}_j\subset\R^3$, $j=1,2,3$. The possible positions of the configuration, together with the unique choice of velocities giving a counterclockwise motion of relative equilibrium, define three 4-dimensional submanifolds $\mathcal{M}_j=\mathcal{B}_j\times\T^2$  of the phase space. Here the projection on the torus $\T^2$ defines the position the triangle: two angles $\theta_1$ and $\theta_2$ locate the respective axis of inertia in each $\R^2$. These angles are defined for these particular motions, but not for general motions, since in general the plane of the triangle does not cut each $\R^2$ along a line, but only at the origin (the centre of mass).
We recall the classical formulas for the angular momentum eigenvalues: $\mu_1=\Theta_1\dot\theta_1$, $\mu_2=\Theta_2\dot\theta_2$, with $\omega_1=\dot\theta_1$, $\omega_2=\dot\theta_2$.  

\begin{proposition}\label{pro:Dar}
If $\mu_1,\mu_2,\theta_1,\theta_2$ are local coordinates on an open set $\mathcal{U}\subset \mathcal{M}_j$, $j=1,2$ or $3$, they are Darboux coordinates, i.e., the canonical symplectic form $\Phi$ of the 3-body problem, restricted to $\mathcal{U}$, is expressed as
$\Phi|_{\mathcal{U}}=d\mu_1\wedge d\theta_1+d\mu_2\wedge d\theta_2$.
\end{proposition}
 
\begin{proof}
We recall that the {\it dual vector fields} of the coordinates $(\mu_1,\mu_2,\theta_1,\theta_2)$ are defined as the vector fields $(Y_1,Y_2,Z_1,Z_2)$ such that $\langle d\mu_i,Y_j\rangle=\delta_{ij}$, where $\delta$ is the Kronecker $\delta$, $\langle d\theta_i,Z_j\rangle=\delta_{ij}$, $\langle d\mu_i,Z_j\rangle=0$, $\langle d\theta_i,Y_j\rangle=0$.

Then, $(\mu_1,\mu_2,\theta_1,\theta_2)$ are Darboux coordinates if and only if after restriction to $\mathcal{U}$ the Hamiltonian vector field of $\mu_i$, $i=1$ or $2$, is $Z_i$ and the Hamiltonian vector field of $\theta_i$ is $-Y_i$.
The Hamiltonian flow of $\mu_1$ is the rotation in the first plane $\R^2$. The Hamiltonian flow of $\mu_2$ is the rotation in the second plane $\R^2$. The flow of $\mu_i$ leaves $\mathcal{M}_j$ invariant. There, it increases $\theta_i$ uniformly in such a way that $Z_1$ is the Hamiltonian vector field of $\mu_1$ and $Z_2$ is the Hamiltonian vector field of $\mu_2$.
The corresponding characterisation of the Hamiltonian vector field of $\theta_i$ cannot be obtained in such an easy way. 

Instead, we compute the symplectic form $\Phi$ of the phase space on the pairs of dual vector fields. We already have 5 coefficients among 6 by writing $\Phi(Z_i,X)=-\langle d\mu_i,X\rangle$, which gives the ``canonical'' answer if $X$ is one of the four dual vector fields: $\Phi(Z_1,Y_1)=\Phi(Z_2,Y_2)=-1$, $\Phi(Z_1,Z_2)=\Phi(Z_1,Y_2)=\Phi(Z_2,Y_1)=0$. The only missing is
$\Phi(Y_1,Y_2)$. By using the torus action it is enough to compute it at $\theta_1=\theta_2=0$. From the classical expression $\Phi=\sum dy_i\wedge dx_i$, we compute
\begin{equation}\label{eqn:AA}
\Phi(Y_1,Y_2)=\sum_{i=1}^{12} \langle dy_i,Y_1\rangle \langle dx_i,Y_2\rangle-\langle dy_i,Y_2\rangle \langle dx_i,Y_1\rangle.
\end{equation}
Here $x_i$, $i=1,..,12$ are all the coordinates of the three bodies, beginning with the first coordinate  $x_1$ of the first body, the second coordinate $x_2$ of the first body, etc., while $y_i$, $i=1,..,12$ are all the coordinates of the momenta, beginning with the first component $y_1$ of the momentum of the first body, etc. The four coordinates of each body and of each momentum are ordered according to the above decomposition $\R^2\oplus\R^2$. 
Thus $\theta_1=0$ gives $x_2=x_6=x_{10}=0$, and, as the velocities are orthogonal to the positions, $y_1=y_5=y_9=0$. Similarly, $\theta_2=0$ gives $x_4=x_8=x_{12}=0$ and $y_3=y_7=y_{11}=0$. When following the flow of $Y_1$ or of $Y_2$, $\theta_1$ and $\theta_2$ are both fixed to zero. We see that each of the terms in the sum (\ref{eqn:AA}) is zero, being the product of two factors, one of them being zero. 
Thus $\Phi(Y_1,Y_2)=0$ and $\Phi|_{\mathcal{U}}=d\mu_1\wedge d\theta_1+d\mu_2\wedge d\theta_2$.\end{proof}

The energy $H$ is defined on $\mathcal{M}_j$ and independent of $\theta_1$ and $\theta_2$. Hamilton's equations on a $\mathcal{U}$ as in Proposition~\ref{pro:Dar} are $\dot\theta_i=\partial H/\partial \mu_i$ and $\dot\mu_i=0$, which imply $\omega_i=\partial H/\partial \mu_i$. This leads to the following proposition.

\begin{proposition}
The slope  for the families of balanced configurations in the image of the scaled energy-momentum map satisfies 
\begin{equation} \label{eqn:slope}
     \frac{dk}{dh} = \frac{\mu_2 - \mu_1}{\omega_1 - \omega_2} (\mu_1 + \mu_2)^{-4}.
\end{equation}
\end{proposition}
\begin{proof}
Consider any family of balanced configurations with a coordinate $s$, which can be $s=b$ when using the $a$-segment 
as described in Lemma~\ref{lem:mono}.
The image in the scaled energy-momentum map of the family is then given by 
the parametric form $(h(s), k(s))$. To find $dk/dh$ we compute the derivatives with respect to the parameter $s$
(denoted by a dash) as
\begin{equation} \label{eqn:dhdk}
    h' = ( \omega_1 - \omega_2) C (\mu_1 + \mu_2), \quad
    k' = (\mu_2 - \mu_1) C (\mu_1 + \mu_2)^{-3}, \quad
    C = \mu_1' \mu_2 - \mu_2' \mu_1
\end{equation}
and forming the ratio proves the proposition.
\end{proof}

\begin{corollary}
1) A smooth maximum in the curve $h(k)$ implies $\omega_1 = \omega_2$ and hence an equilateral configuration.
2) The configuration with round tensor of inertia
is on the long family, at a point with negative slope.
3) A smooth maximum in the curve $k(h)$ implies $\mu_1 = \mu_2$ and hence $k = \tfrac14$.
\end{corollary}
\begin{proof}
1) A smooth maximum in $h$ implies $dh/dk = 0$, and hence by the slope formula $\omega_1 = \omega_2$.
Thus $\Omega^2$ is proportional to the identity, and the configuration is a central configuration.
2) A round tensor of inertia means that $\Theta_1 = \Theta_2 = \Theta$. This implies that $\Theta_1\Theta_2$ is critical when $\Theta_1+\Theta_2$ is fixed. But $A^2=M\Theta_1\Theta_2/4M_3$ and $I=\Theta_1+\Theta_2$. The configuration is at the centre of the ellipses of Fig.~\ref{fig:shapetriangle}, right. The first two rows of matrix~(\ref{eqn:B}) are proportional. Then, $(a,b,c)$ is proportional to $(m_2^{-1}+m_3^{-1},m_3^{-1}+m_1^{-1},m_1^{-1}+m_2^{-1})$. In particular, $(a,b,c)$ is ordered as $(m_1,m_2,m_3)$. The long family is defined as the family passing through the equilateral triangle, and there, according to Lemma~\ref{lem:B}, it passes from the region where $(a,b,c)$ is ordered as $(m_1,m_2,m_3)$ to the region where $(a,b,c)$ is in reversed order. On the short families $(a,b,c)$ are ordered differently, so, the configuration is on the long family. Now, $\mu_i=\omega_i\Theta_i$, $i=1,2$. Hence $(\mu_2-\mu_1)/(\omega_1 - \omega_2) = - \Theta$ so that the slope is negative.
3) A smooth maximum in $k$ implies $dk/dh = 0$, and hence by the slope formula $\mu_1 = \mu_2$. 
This implies that the scaled momentum is $k=\mu \mu / (\mu+ \mu)^2 = \tfrac14$.
\end{proof}

The image of a {\it parametrised path} $\R\to\R^m,\, t\mapsto f(t)$  displays a {\it cusp} at $f(t_0)$ if $f'(t_0)=0$ while $f''(t_0)\neq 0$. The vector $f''(t_0)$ indicates the direction of the cusp. In the case where $f'(t_0)=0$ and $f''(t_0)=0$, the image may have a different aspect, and may even appear as perfectly smooth, in particular, if the singularity at $t_0$ disappears after a change of parameter. Here we are interested in the images of the three smooth curves of balanced configurations by the map $(h,k)$. We will speak of a cusp when $(h',k')=(0,0)$, even if also $(h'',k'')=(0,0)$. We take as parameter a coordinate of the smooth curve of balanced configurations.

\begin{lemma}
A cusp in the $(h,k)$ diagram occurs when $C = \mu_1' \mu_2 - \mu_2' \mu_1 = 0$.
\end{lemma}
\begin{proof}
Equation \eqref{eqn:dhdk} says that when $C$ vanishes both derivatives vanish, $h' = k' = 0$.
\end{proof}
Note that $C = 0 $ can be interpreted as the condition  for the ratio $\mu_1 / \mu_2$ to have a critical point on the curve of balanced configurations. Then, $h$ and $k$ also have a critical point at this same point.
The dimensionless momentum $k=\mu_1\mu_2/(\mu_1+\mu_2)^2$ is a function of $\mu_1 / \mu_2$ which has the advantage of being invariant under exchange of $\mu_1$ and $\mu_2$.

Note also that the local coordinate hypothesis of Proposition~\ref{pro:Dar} implies that $d\mu_1\wedge d\mu_2\neq 0$ on the open set $\mathcal{U}$. But $d\mu_1\wedge d\mu_2=0$ on $\mathcal{M}_i$ if and only if $d\mu_1\wedge d\mu_2=0$ on $\mathcal{B}_i$. This is a codimension 1 and homogeneous condition on the semi-cone $\mathcal{B}_i$. Let $X$ be the non-zero ``velocity vector'' defined by the parametrisation by $s$ of the curve of balanced configuration. The condition $X\rfloor(d\mu_1\wedge d\mu_2)=0$ obtained by contraction of $X$ is equivalent to $d\mu_1\wedge d\mu_2=0$, and is also $\mu_1'd\mu_2-\mu_2'd\mu_1=0$. By contracting now a radial vector  and using the homogeneity of $(\mu_1,\mu_2)$, we see that this condition is $C=0$. Consequently \emph{a cusp happens when the local coordinate hypothesis of Proposition~\ref{pro:Dar} fails to be satisfied.} Equation (\ref{eqn:dhdk}) is however still valid at $C=0$ since it may be obtained by passing to the limit as $C$ tends to zero.

\begin{proposition}\label{pro:cusp}
Every balanced family has a cusp or it touches the maximum $k = 1/4$, or both simultaneously.
\end{proposition}
\begin{proof}
The dimensionless momentum $k$ goes to zero at both endpoints of the family according to Lemma~\ref{lem:C}.
Otherwise $k$ is positive and bounded above by $1/4$, so it must have a maximum somewhere. 
At the maximum $k' = 0$, and so either $\mu_1 = \mu_2$ and hence $k = 1/4$, or $C = 0$, or both. 
When $C = 0$ then also $h' = 0$ and there is a cusp. When $\mu_1 = \mu_2$ but $C \not = 0$ 
there is a smooth tangency at $k = 1/4$. When both factors vanish there is a cusp at $k = 1/4$.
\end{proof}

The special case where both, $\mu_1 = \mu_2$ and $C=0$ occur simultaneously, 
so that there is a cusp at $k=1/4$ does take place when all masses are equal, see Fig.~\ref{fig:equalmasses}.

If a cusp satisfies the non-degeneracy condition $(h'',k'')\neq (0,0)$, the slope $k'/h'$ varies smoothly as a function of the parameter $s$ when passing the cusp, as shown by the Taylor expansion of the numerator and the denominator. In particular, it does not change sign, except maybe at the already considered special points where the slope is zero or infinity. If we start from the Lagrange point, where the slope is infinity, and increase $k$, the slope is negative. Even if we meet a non-degenerate cusp, the slope will remain negative.
But we are on the long family, and should follow it, in the cases of Figs.~\ref{fig:gen1} or \ref{fig:2equ12}, up to a collision configuration.
Indeed, from \cite{DS19}, we know that near the collision configuration where $h \to -\infty$ and $k\to 0$ the slope $dk/dh$ is positive. The slope should change sign and consequently, under non-degeneracy assumption, {\it the long family should have a configuration with $k = 1/4$}, which is verified on Figs.~\ref{fig:gen1} and \ref{fig:2equ12}, but also on Fig.~\ref{fig:2equ2}, for which this fact will be established in the next section.
Apparently, we always have existence, but also uniqueness of the balanced configuration with $k=1/4$.

\begin{conjecture}
For any choice of masses, a unique balanced configuration allows a motion of relative equilibrium with a given angular momentum $L$ with equal eigenvalues $\mu_1=\mu_2$.
\end{conjecture}

This conjecture would imply that the short families do not reach $k=1/4$, and, by Proposition~\ref{pro:cusp}, that each of them has at least a cusp. We could also deduce that the slope is always positive on the short families, even if there are several cusps. But apparently, there is exactly one cusp on each short family, and there are no cusps on the long family. This would be a corollary of the following conjecture.

\begin{conjecture}
On each balanced family, the functions $h$ and $k$ have exactly one critical point each.
\end{conjecture}

This conjecture implies that for generic angular momentum there are either 2, 4 or 6 
non-Lagrange relative equilibria. The smallest number 2 occurs for large $k$ close to $1/4$, while the largest number 6 occurs for small $k$ near 0.

\section{Special masses}

We are going to present analytical results for three qualitatively different cases:
three equal masses and two equal masses with the third mass either smaller 
or larger than the equal masses.

\subsection{Three Equal masses}

\begin{figure} 
\includegraphics[height=6cm]{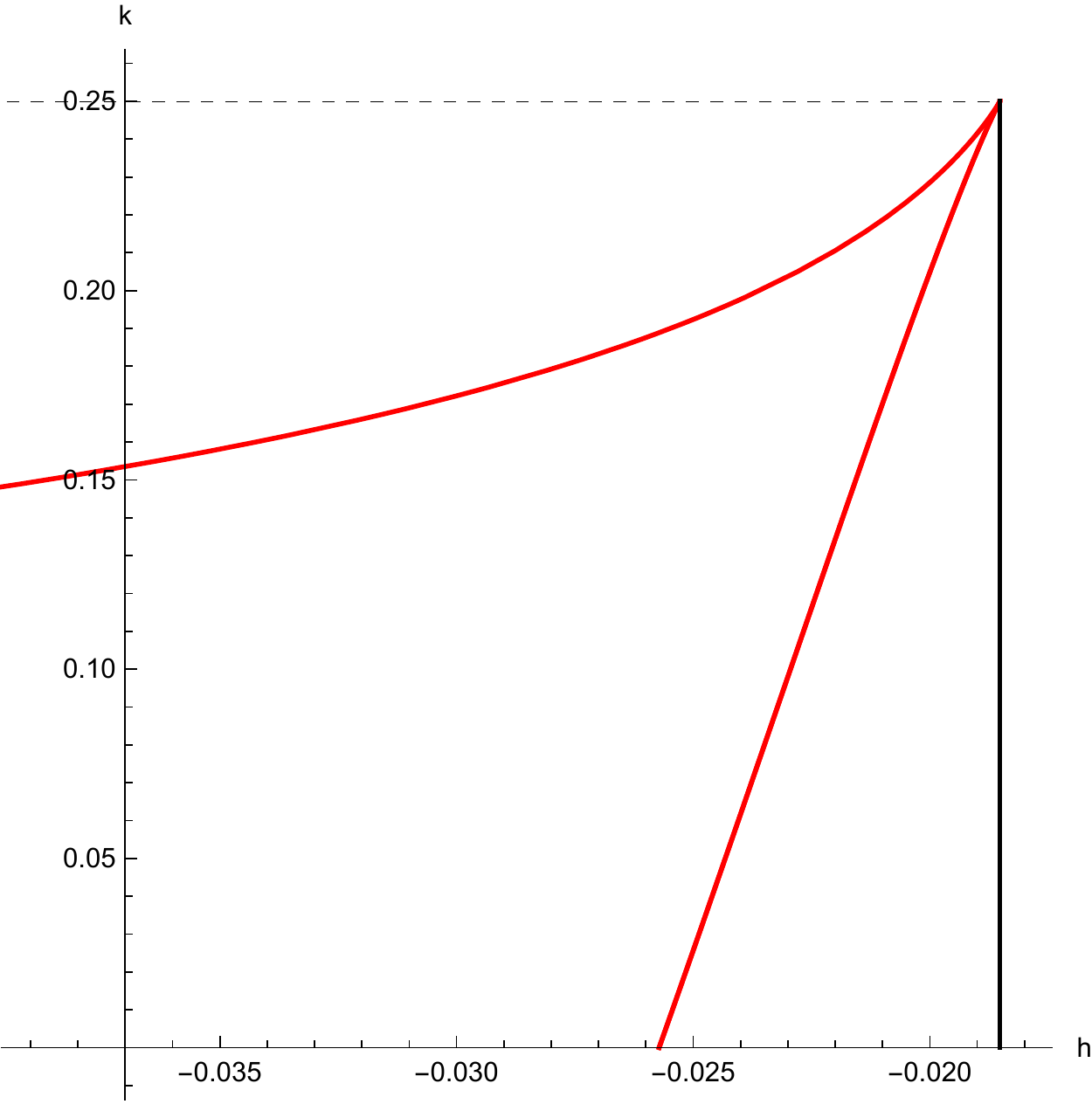}
\includegraphics[height=6cm]{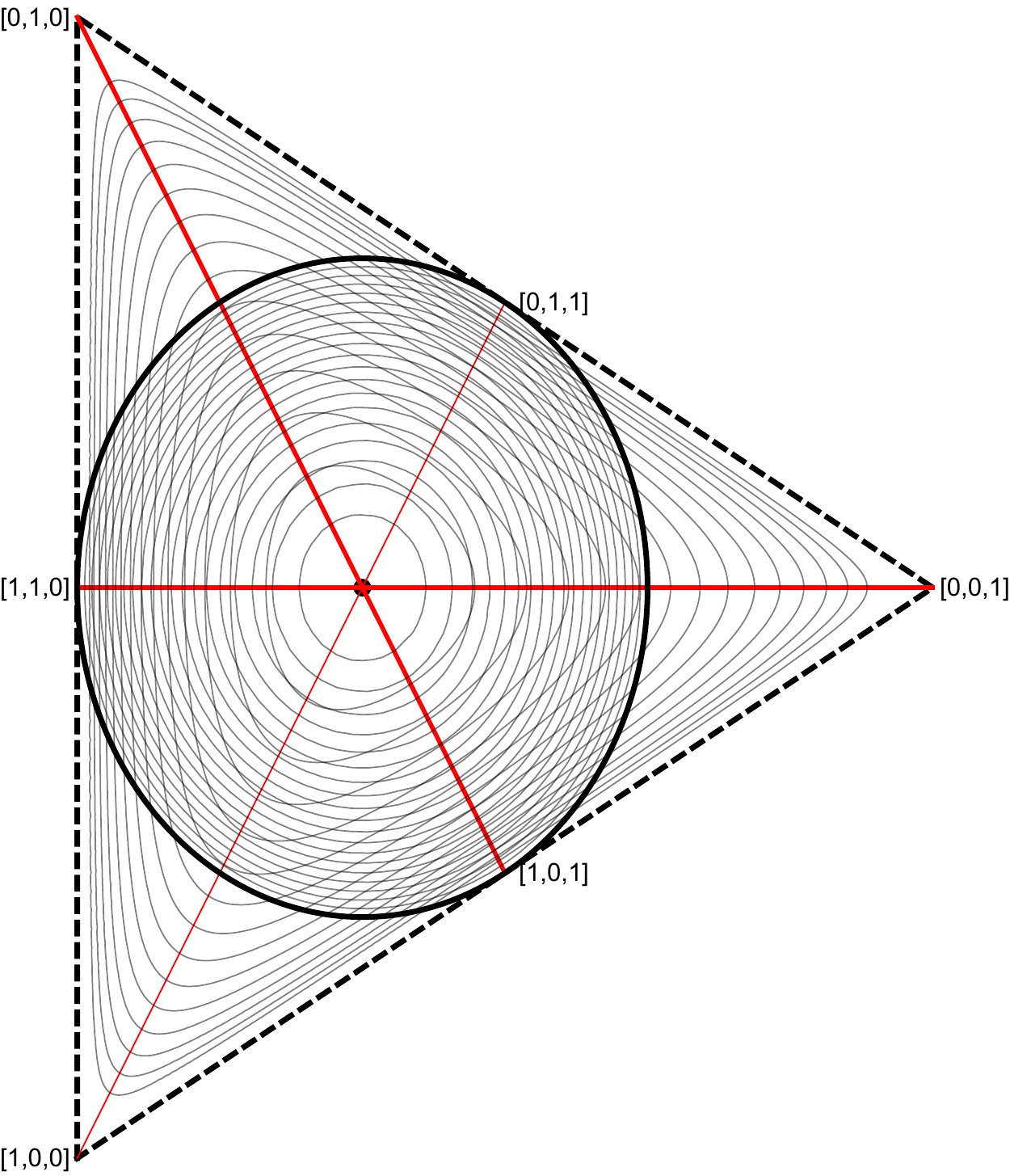}
\caption{Equal mass case. The Lagrange equilateral family is the vertical line, in this case extending 
all the way to $k = 1/4$. The three isosceles families coincide and emerge 
at the collinear Euler solution at $k = 0$. Isosceles triangles with $\rho < 1$ 
are non-linearly stable because they are minima in the energy for fixed $k$.} \label{fig:equalmasses}
\end{figure}

When the masses are all equal any isosceles triangle is a balanced configuration, 
and no other triangle is balanced, see Fig.~\ref{fig:equalmasses}. 
The three solution branches shown on the right are equivalent up to permutation of masses.
In the left figure the top branch of minimal energy at fixed 
angular momentum invariant $k$ corresponds to a definite Hessian. 
Explicit formulas for this case can be obtained from the next theorem 
for the special value of the mass ratio parameter $\mu = 1$.

\goodbreak

\subsection{Two equal masses}

\begin{figure} 
\includegraphics[height=6cm]{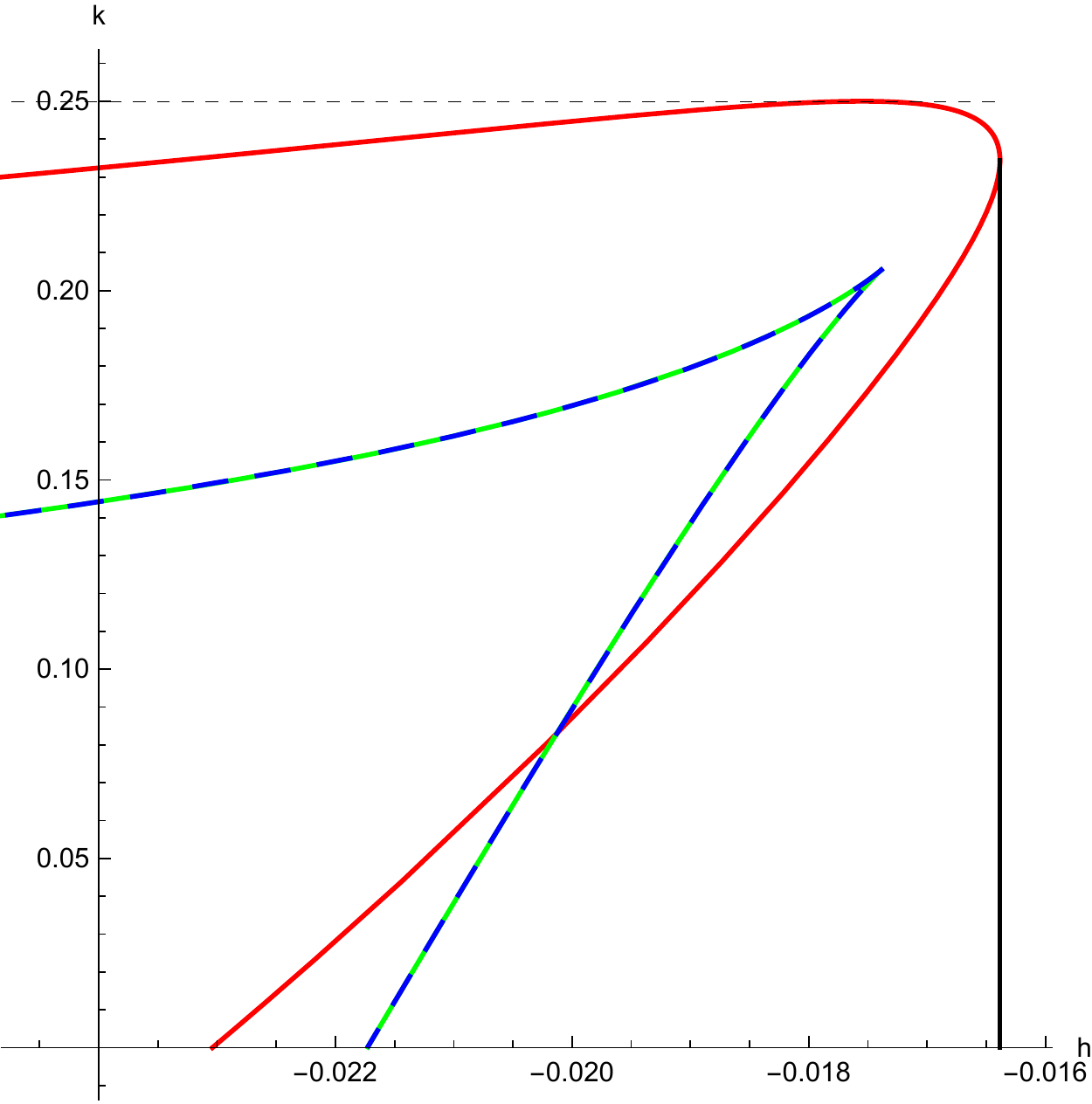}
\includegraphics[height=6cm]{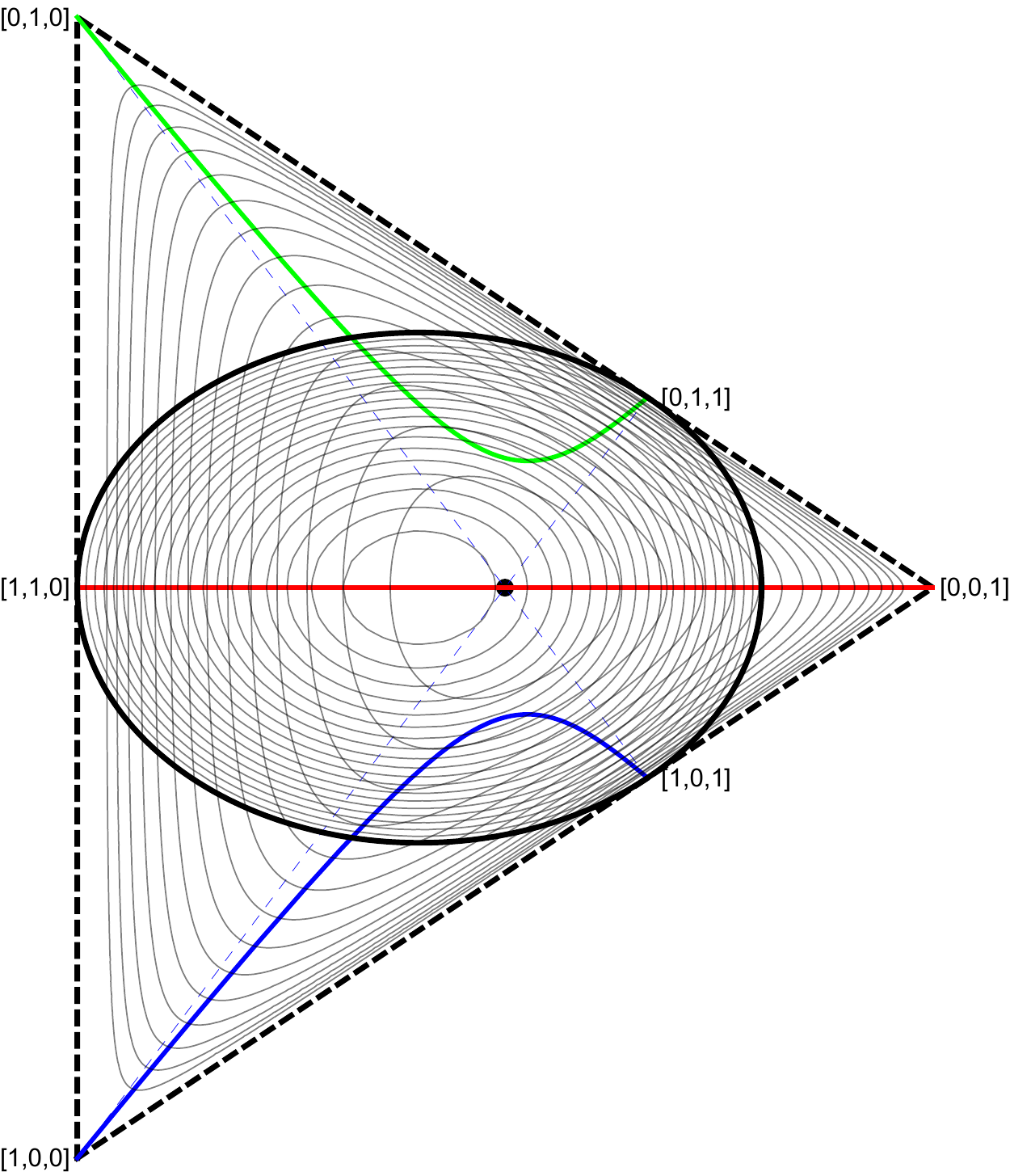}
\caption{Two equal masses, third mass smaller ($\mu = 1/2$).
The Lagrange equilateral family is the vertical line, {\em not} extending 
all the way to $k = 1/4$. At the endpoint it meets the isosceles long family, 
which later touches $k=1/4$. Both short families of asymmetric triangles
emerge from a collinear Euler configuration at $k=0$ and have a cusp
beyond which $h$ approaches $-\infty$. The isosceles configurations to the 
left of the tangency with $k = 1/4$ are the absolute minimum of the energy 
and hence are non-linearly stable.
} \label{fig:2equ12}
\end{figure}

\begin{figure} 
\includegraphics[height=5cm]{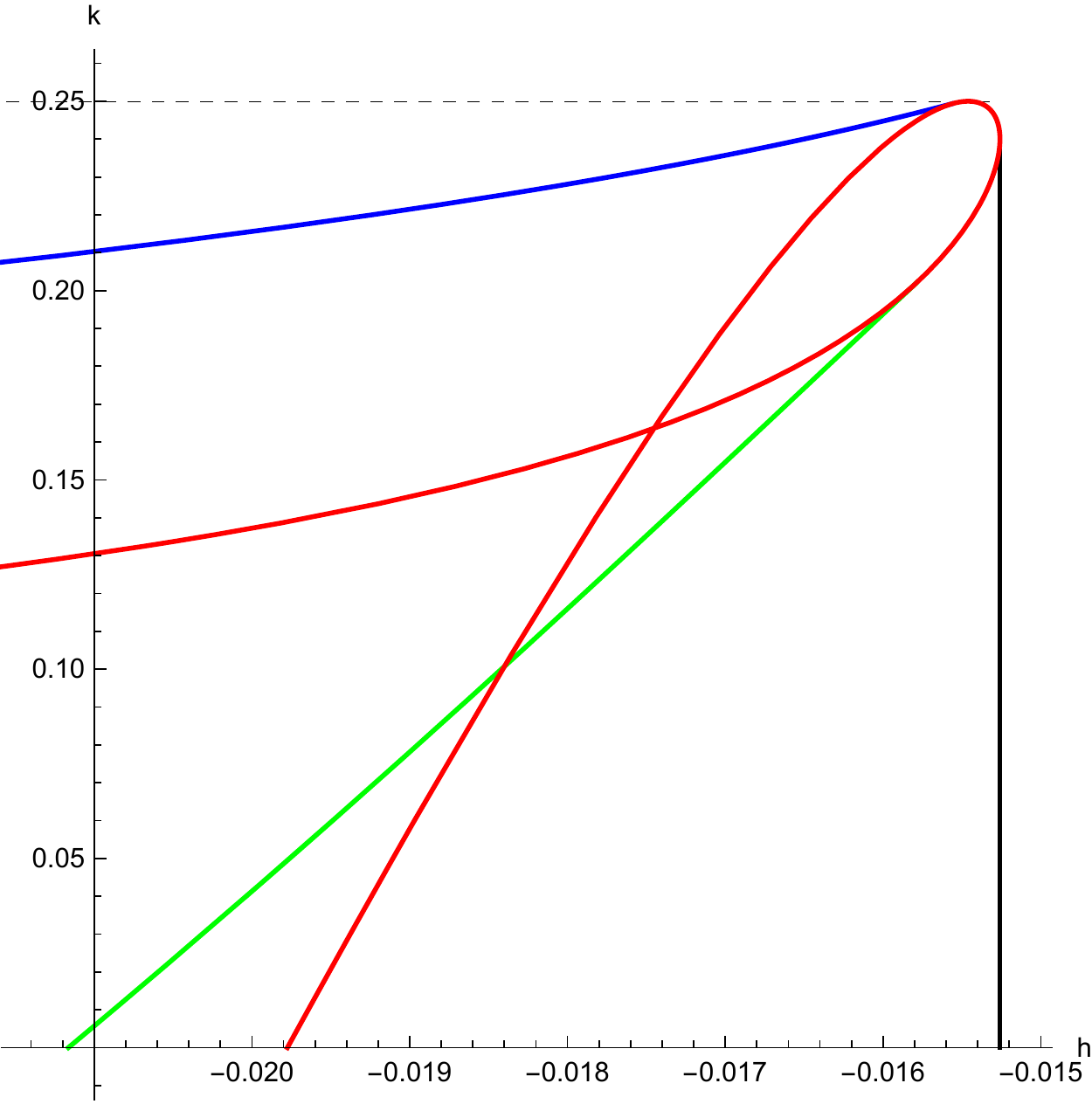}
\includegraphics[height=5cm]{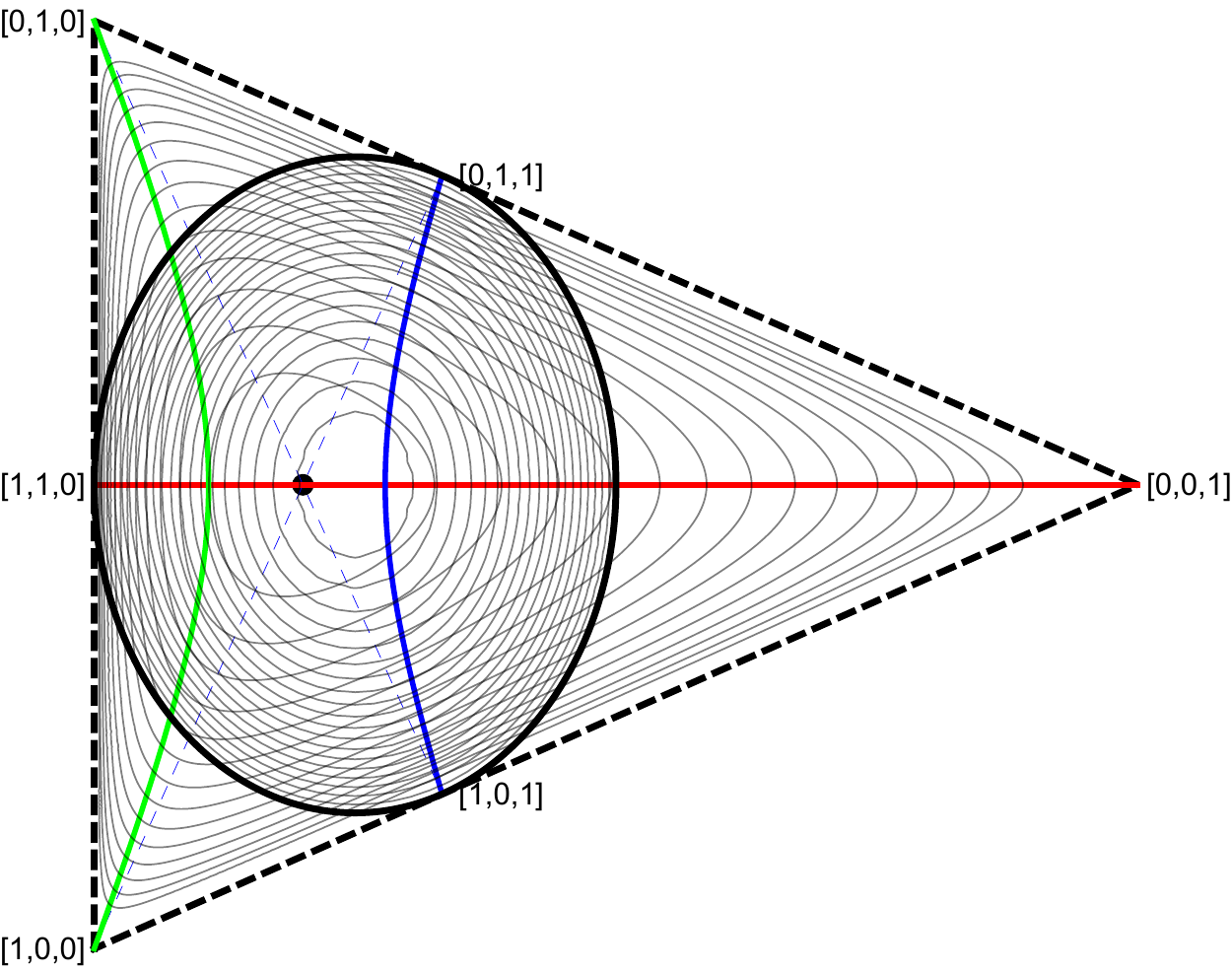}
\caption{Two equal masses, third mass bigger ($\mu = 2$).
The Lagrange equilateral family is the vertical line, not extending 
all the way to $k = 1/4$. At the endpoint it meets the long isosceles family (red), 
which later touches $k=1/4$. One short family (green) of asymmetric triangles
emerges from a collinear Euler configuration at $k=0$, has a cusp 
tangent to the long family and retraces itself back down.
The other short family  (blue) of asymmetric triangles starts 
and finishes at the collision where $h \to -\infty$, 
and has a cusp tangent to the long family.
These asymmetric triangles of absolute minimal energy are non-linearly stable.
There is a tiny part of the long family of symmetric isosceles triangles which have 
absolute minimal energy and hence are non-linearly stable.
} \label{fig:2equ2}
\end{figure}

When two (but not all) masses are equal two cases need to be distinguished,
depending on whether the single mass is smaller or larger than the multiple masses. 
Denote the ratio of the single mass to that of the repeated mass by $\mu$, then 
we need to distinguish $0 < \mu < 1$ and $1 < \mu$. The shape of the isosceles 
triangle is described by $\rho$, which is the ratio of the length of the single side to 
length of the repeated side. For obtuse triangles $\sqrt{2} < \rho \le 2$ where $\rho = 2$
is the limiting collinear case and $\rho = \sqrt{2}$ is an isosceles right triangle. 
Triangles with $0 < \rho < \sqrt{2}$ are acute where $\rho = 1$ is equilateral and 
the limiting case $\rho \to 0$ corresponds to collision.
The isosceles configuration can be analytically analysed. Notice, however, that 
in the case of two equal masses there are also two additional families of non-isosceles 
triangles.

\begin{theorem}
For masses $(m_1, m_2, m_3) = (1,1,\mu) m$ there is a family of isosceles triangles with $(d_{23}, d_{13} , d_{12}) = (1,1,\rho) s$
that is a balanced configuration for some length scaling factor $s$ and mass scaling factor $m$.
The image of this family in the energy-momentum map is given by 
\[
     h(\rho)   =
    -\frac18 m^5 (1 + 2 \rho \mu)( 2 + \rho^3 \mu) (1 + \chi(\rho))^2,\qquad k(\rho)=\frac{1}{2  + \chi(\rho) +  \chi(\rho)^{-1}}, 
\]
where 
\[     
      \chi(\rho) = 
      \frac{ ( 4 - \rho^2) \mu}{ \sqrt{\rho ( 2 + \mu) ( 2 + \rho^3 \mu)} } 
\]
parametrised by $0 < \rho < 2$.
\end{theorem}

Before proving this result we describe the resulting energy-momentum diagrams for the 
two cases with mass ratio $\mu < 1$ and $\mu > 1$.

When the non-equal mass is smaller, see Fig.~\ref{fig:2equ12} for an example with $\mu = 1/2$, 
the  branch of the isosceles family  that has  minimal energy for fixed angular momentum $k$
extends until the tangency at $k =1/4$.
The loss of definiteness at the tangency with $k = 1/4$ is somewhat surprising since it occurs 
by driving an eigenvalue through infinity.

When the non-equal mass is bigger, see Fig.~\ref{fig:2equ2} for an example with  $\mu = 2$,
the situation is quite different. First of all in this case it is convenient to order the 
masses as $(\mu, 1, 1)$ and use the $c$-segment to obtain the smooth families
of balanced equilibria. Note that when this case is perturbed there is no 
continuity in the families, compare Fig.~\ref{fig:gen23}.
By contrast for $\mu < 1$ there is continuity in all three families.
For $\mu > 1$ the branch of the isosceles family that goes to $-\infty$  is not a global minimum.
For most values of $k$ the global minimum instead occurs on the asymmetric branch (blue).
This branch attaches in a tangency of its cusp to the isosceles branch very close to the top-tangency 
at $k=1/4$. Accordingly, there is a tiny branch of the isosceles branch which is 
in fact the absolute minimum. This occurs between 
the attachment point of the short family and the tangency at $k=1/4$.

For $\rho = 1$ the curve of isosceles triangles
in the $(h,k)$ diagram touches the endpoint of the Lagrange family at the 
point $(3 \mu(2 + \mu)/( 4 (1 + 2\mu)^2), -m^5( 1 + 2 \mu)^3/( 2 ( 2 + \mu)) )$, 
compare with Theorem 1.
This is a very special case of a general result by Chenciner about possible bifurcations \cite{Chenciner11}.

\begin{proof}
With the overall scaling factor $s$ let  $\xi_i \in \R^4$ be given by
\[
    \xi_1 = (0, -\gamma(\rho), 0, -1/2) \rho s, \quad
    \xi_2 = (0, -\gamma(\rho), 0, 1/2) \rho s, \quad
    \xi_3 = (0, 2\gamma(\rho)/\mu, 0, 0) \rho s,
\]
where 
\[
    \gamma(\rho) = \frac{  \sqrt{ 4 \rho^{-2} - 1} }{2 ( 1 + 2 \mu^{-1})} \,.
\]
This implies $d_{23} = d_{13} = s$ and $d_{12} = \rho s$ describing an isosceles triangle.
The non-zero eigenvalues of the diagonal tensor of inertia $S=\sum m_i \xi_i \xi_i^t$ 
are $\Theta_1 = m s^2 \mu ( 4 - \rho^2) / ( 2 ( 2 + \mu))$ and $\Theta_2 = \tfrac12 m s^2 \rho^2$
such that $4 \gamma^2 ( 2 + \mu) = \mu \Theta_1 / \Theta_2$.

The frequency matrix satisfies $\Omega^2 = \mathrm{diag}( -\omega_1^2 Id, -\omega_2^2 Id)$
with scalar frequencies
\[
\omega_1^2 = m ( \mu + 2) U'(s)/s, \quad
\omega_2^2 = m( \mu U'(s)/s + 2 U'(\rho s) / (\rho s)) 
\]
where $U(r) = -1/r$ is the two-particle potential.
With these definitions it is straightforward to check that indeed  equation \eqref{eqn:releq}
and its cyclic permutations are satisfied.
The  angular momentum tensor corresponding to these balanced configurations is block-diagonal and since it is 
anti-symmetric there are only two independent non-zero entries given by 
$L_{12} =  \Theta_1 \omega_1$ and $L_{34} =  \Theta_2 \omega_2$.
From these the invariants of $L$ can be computed, and the resulting scaled momentum is
\[
   k^{-1} = \frac{ ( L_{12} + L_{34})^2 } { L_{12} L_{34} }
      =  2  + \chi +  \chi^{-1}, \quad
      \chi = \frac{\omega_1}{\omega_2} \frac{\Theta_1}{\Theta_2} \,.
\]
It is easy to see that $d \chi / d\rho < 0$ and hence $\chi = 1$ implies $k = 1/4$ and that this is 
a non-degenerate maximum in $k$.
Finally the Hamiltonian is $H  = T + V$ where 
\[
   T = \frac12  ( \Theta_1 \omega_1^2  + \Theta_2 \omega_2^2), \quad
   V = - m^2 \frac{ 1 + 2 \rho \mu}{\rho s}  \,.
\]
The identity $2 T + V = 0$ holds at these balanced configurations. 
Thus we find
\[
   h = -\frac18 m^5 (1 + 2 \rho \mu)( 2 + \rho^3 \mu) (1 + \chi)^2  
\]
and this proves the result.
\end{proof}

\begin{figure} 
\includegraphics[height=6cm]{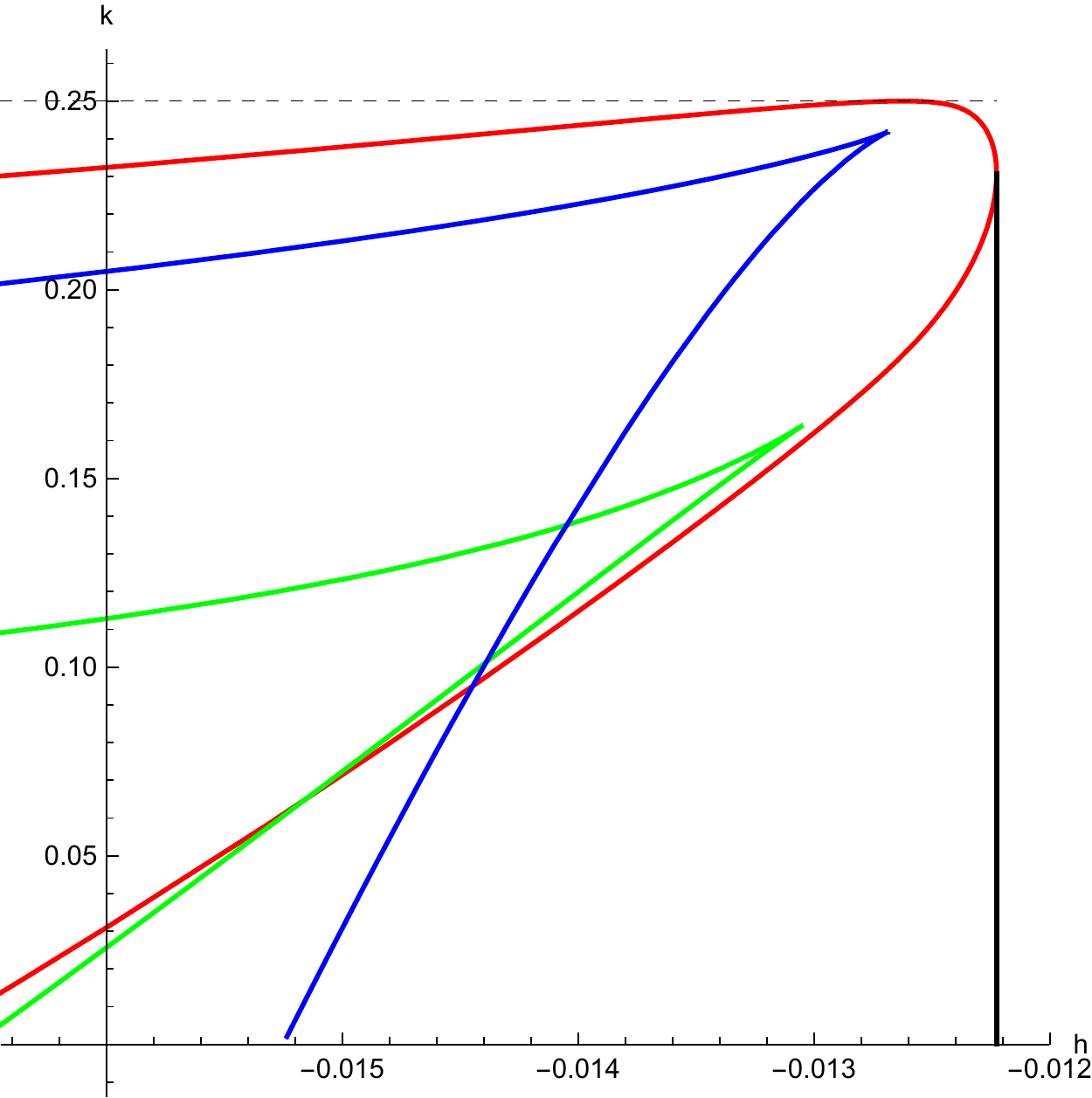}
\includegraphics[height=6cm]{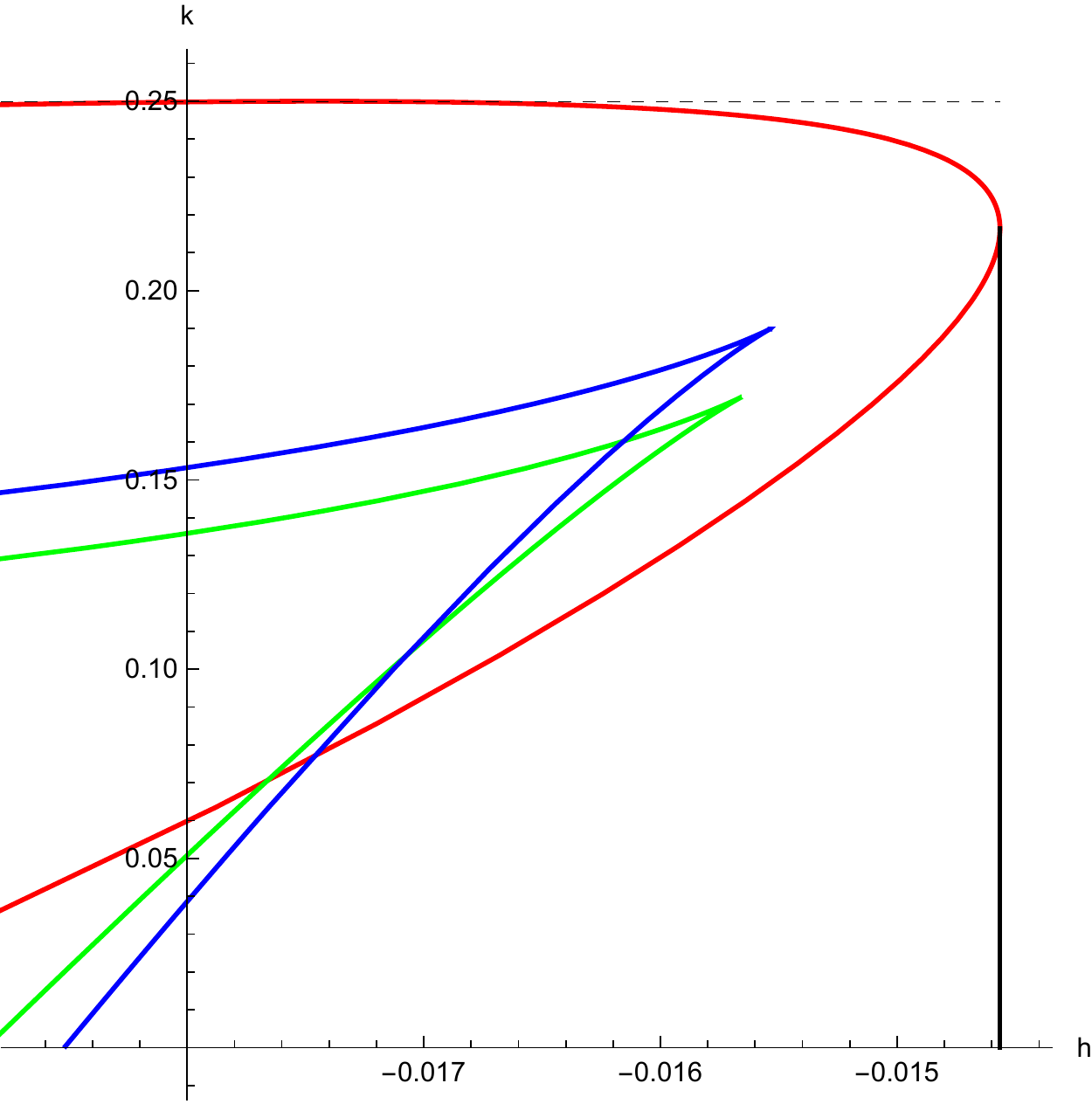}
\caption{Three distinct masses, somewhat close to the two isosceles cases.
Left: masses $(12, 5, 4)/21$, Right: masses $(6,5,2)/13$. 
The left figure illustrates that there is no continuity in the balanced families 
when perturbing from the case with two equal masses and the third mass larger 
than the equal ones, compare Fig.~\ref{fig:2equ2}.
} \label{fig:gen23}
\end{figure}

\section{Acknowledgement}

Extensive discussions with Rick Moeckel in 2015 at the Observatory in Paris that lead to the initial results of this paper 
are gratefully acknowledged. We also thank Alain Chenciner and James Montaldi for fruitful discussions at the same time. We would like to thank the anonymous referees for their comments which helped to improve the manuscript.

\appendix

%\bibliographystyle{AIMS}      % mathematics and physical sciences
%\bibliography{../../bib_cv/all,../../bib_cv/hd}{}   % name your BibTeX data base

\begin{thebibliography}{10}

\bibitem{Albouy04}
\newblock A.~Albouy,
\newblock Mutual {D}istances in {C}elestial {M}echanics, {L}ectures at {N}ankai
  institute, {T}ianjin, {C}hina,
\newblock \emph{preprint} (2004).

\bibitem{ACS12}
\newblock A.~Albouy, H.~E. Cabral and A.~A. Santos,
\newblock Some problems on the classical $n$-body problem,
\newblock \emph{Celestial Mechanics and Dynamical Astronomy}, \textbf{113}
  (2012), 369--375.

\bibitem{AlbouyChenciner98}
\newblock A.~Albouy and A.~Chenciner,
\newblock Le probl{\`e}me des {$n$} corps et les distances mutuelles,
\newblock \emph{Invent. Math.}, \textbf{131} (1998), 151--184.

\bibitem{Chenciner11}
\newblock A.~Chenciner,
\newblock The angular momentum of a relative equilibrium,
\newblock \emph{Discrete {\&} Continuous Dynamical Systems - A}, \textbf{33}
  (2013), 1033--1047.

\bibitem{Chenciner13}
\newblock A.~Chenciner and H.~Jim{\'e}nez-P{\'e}rez,
\newblock Angular momentum and {H}orn's problem,
\newblock \emph{Mosc. Math. J.}, \textbf{13} (2013), 621--630.

\bibitem{Dullin13}
\newblock H.~R. Dullin,
\newblock The {L}ie-{P}oisson structure of the reduced $n$-body problem,
\newblock \emph{Nonlinearity}, \textbf{26} (2013), 1565--1579.

\bibitem{DS19}
\newblock H.~R. Dullin and J.~Scheurle,
\newblock Symmetry reduction of the 3-body problem in {$R^4$},
\newblock \emph{Journal of Geometric Mechanics}, (in print, accepted 20 Nov  2019), arXiv:1908.04496.

\bibitem{Herman98}
\newblock M.~Herman,
\newblock Some open problems in dynamical systems,
\newblock in \emph{Proceedings of the International Congress of
  Mathematicians}, vol.~2 of Documenta Mathematica,
\newblock DMV, 1998,
\newblock 797--808.

\bibitem{Lagrange1788}
\newblock J.~L. Lagrange,
\newblock \emph{M{\'e}chanique analitique},
\newblock Paris, 1788.

\bibitem{Moeckel17}
\newblock R.~Moeckel,
\newblock Minimal energy configurations of gravitationally interacting rigid
  bodies,
\newblock \emph{Celestial Mechanics and Dynamical Astronomy}, \textbf{128}
  (2017), 3--18.

\bibitem{Scheeres12}
\newblock D.~J. Scheeres,
\newblock Minimum energy configurations in the {$N$}-body problem and the
  celestial mechanics of granular systems,
\newblock \emph{Celestial Mechanics and Dynamical Astronomy}, \textbf{113}
  (2012), 291--320.

\bibitem{Sundman12}
\newblock K.~F. Sundman,
\newblock M{\'e}moire sur le probl{\`e}me des trois corps,
\newblock \emph{Acta mathematica}, \textbf{36} (1913), 105--179.

\bibitem{Wintner41}
\newblock A.~Wintner,
\newblock \emph{The analytical foundations of celestial mechanics},
\newblock Princeton University Press, 1941.

\end{thebibliography}

\def\cprime{$'$}
\providecommand{\href}[2]{#2}
\providecommand{\arxiv}[1]{\href{http://arxiv.org/abs/#1}{arXiv:#1}}
\providecommand{\url}[1]{\texttt{#1}}
\providecommand{\urlprefix}{URL }

\end{document}